\documentclass{amsart}
\usepackage{amssymb}
\usepackage[all]{xy}
\usepackage[pdftex]{graphics,color}

\newtheorem{prop}{Proposition}[section]
\newtheorem{lemma}[prop]{Lemma}
\newtheorem{defin}[prop]{Definition}
\newtheorem{cor}[prop]{Corollary}
\newtheorem{theorem}[prop]{Theorem}
\newtheorem{claim}[prop]{Claim}

\newcounter{cor4}
\newcounter{lemma1}

\newcounter{claim1}

\newcounter{claim3}
\newcounter{claim4}
\newcounter{claim}

\newcommand{\ho}{\mathcal{O}}

\newcommand{\altA}{\mathcal{A}}
\newcommand{\altB}{\mathcal{B}}
\newcommand{\altD}{\mathcal{D}}
\newcommand{\altT}{\mathcal{T}}

\newcommand{\altE}{\mathcal{E}}
\newcommand{\altH}{\mathcal{H}}

\newcommand{\altL}{\mathcal{L}}

\newcommand{\altM}{\mathcal{M}}
\newcommand{\altN}{\mathcal{N}}
\newcommand{\altP}{\mathcal{P}}
\newcommand{\altQ}{\mathcal{Q}}

\newcommand{\sur}{\twoheadrightarrow}

\newcommand{\roof}[5]{\xymatrix @R=0.5cm @C=0.5cm { & {#2} \ar[dr]^{#5} \ar[dl]_{#4} & \\ {#1}  & & {#3}}}

\newcommand{\dd}{^{\vee\vee}}

\newcommand{\Dr}[1]{\DER^{b,#1}}
\newcommand{\Ch}[1]{\Coh^{#1}}
\newcommand{\QCh}[1]{\altM od^{#1}}

\DeclareMathOperator{\Hom}{Hom}

\DeclareMathOperator{\Ext}{Ext}
\DeclareMathOperator{\stab}{Stab}

\DeclareMathOperator{\SL}{SL}

\DeclareMathOperator{\Aut}{Aut}
\DeclareMathOperator{\Coh}{Coh}
\DeclareMathOperator{\QCoh}{\altM od}
\DeclareMathOperator{\Cohc}{Coh_{(c)}}
\DeclareMathOperator{\Qcohc}{QCoh_{(c)}}
\DeclareMathOperator{\QCohc}{\altM od_{(c)}}
\DeclareMathOperator{\Derc}{D^b_{(c)}}
\DeclareMathOperator{\QDerc}{D^+_{(c)}(\altM od(X))}
\DeclareMathOperator{\Ima}{Im}
\DeclareMathOperator{\Rea}{Re}

\DeclareMathOperator{\im}{im}
\DeclareMathOperator{\rk}{rk}
\DeclareMathOperator{\Id}{id}
\DeclareMathOperator{\DER}{D}
\DeclareMathOperator{\Der}{D^b}

\DeclareMathOperator{\supp}{supp}
\DeclareMathOperator{\codim}{codim}
\DeclareMathOperator{\Pic}{Pic}
\DeclareMathOperator{\coker}{coker}
\DeclareMathOperator{\ch}{ch}
\DeclareMathOperator{\td}{td}
\DeclareMathOperator{\ce}{c}
\DeclareMathOperator{\Ho}{H}
\DeclareMathOperator{\Kom}{Kom^b}
\DeclareMathOperator{\torssh}{t}
\DeclareMathOperator{\cl}{cl}
\DeclareMathOperator{\Glt}{\widetilde{GL}{}^+(2,\mathbb{R})}
\DeclareMathOperator{\Gl2}{GL^+(2,\mathbb{R})}

\DeclareMathOperator{\depth}{depth}
\DeclareMathOperator{\dimh}{dh}
\DeclareMathOperator{\Amp}{Amp}
\DeclareMathOperator{\Nef}{Nef}
\DeclareMathOperator{\Eff}{Eff}
\DeclareMathOperator{\Weil}{Weil}
\DeclareMathOperator{\Int}{Int}
\DeclareMathOperator{\HOM}{HOM}
\DeclareMathOperator{\EXT}{EXT}
\DeclareMathOperator{\Ka}{K}
\DeclareMathOperator{\Vect}{Vect}
\DeclareMathOperator{\Num}{N}
\DeclareMathOperator{\tot}{tot}
\DeclareMathOperator{\Bir}{Bir}
\DeclareMathOperator{\divisor}{div}
\DeclareMathOperator{\Var}{\mathcal{V}ar}

\title{Quotient categories, stability conditions, and birational geometry}
\author{Sven Meinhardt, Holger Partsch}

\begin{document}

\maketitle

\begin{abstract}
This article deals with the quotient category of the category of coherent sheaves on an irreducible smooth projective variety by  the full subcategory of sheaves supported in codimension greater than c. It turns out that  this category has homological dimension c. As an application of this, we will describe the space of stability conditions on its derived category in the case c=1. Moreover, we describe all exact equivalences between these quotient categories in this particular case which is closely related to classification problems in birational geometry. 
\end{abstract}

\tableofcontents

\section{Introduction}

After stability conditions were invented by T. Bridgeland \cite{Bridgeland02} a lot of effort has been done to describe the space of stability conditions for various situations in algebra and geometry. In complex geometry this space is more or less well understood for curves (\cite{Macri},\cite{Okada}), projective and generic K3 surfaces resp.\ two-dimensional tori (\cite{Bridgeland03},\cite{HuybStelMac}) and for generic complex tori (\cite{Meinhardt07}). Moreover, there are some `exceptional' cases were stability conditions are known (\cite{Bridgeland06},\cite{Bridgeland05},\cite{BridgelandCY},\cite{Macri04},\cite{MacMehStel},\cite{Meinhardt}). Nevertheless, a general approach to construct stability conditions on a variety is far of reach. In particular, there is no known stability condition on a compact Calabi--Yau threefold yet, which is very unsatisfying as the theory of stability conditions was motivated by string theory. Roughly speaking, as long as the $\Ka$-group and the homological dimension are small, there is a chance to make some progress. \\
Having this in mind, one could study the space of stability conditions on triangulated subcategories or quotient categories of $\Der(X)$, where $X$ is a sufficiently `nice' variety. The results might give some hints to the situation on the whole derived category $\Der(X)$. \\
This article was motivated by the attempt to describe spaces of stability conditions on quotient categories. It deals with the quotient category $\Derc(X)$ of $\Der(X)$ obtained by sending all complexes supported in codimension greater than $c$ to zero. It turns out that this category is the bounded derived category of an abelian category $\Cohc(X)$ obtained in the same way (see Appendix A). While writing this article the connection between these quotient categories and questions of birational geometry became more and more transparent and we devote the last part of this article to this relationship. There is a very nice equivalent description of this quotient category which throws a first light on its connection to birational geometry.
\begin{prop} For any smooth projective variety $X$ over the field $k$ there is an exact $k$-linear equivalence
\[ \Derc(X) \cong \varinjlim_{\codim(X\setminus U)>c} \Der(U),\]
where the direct limit is taken over the directed set of open subsets $U\subset X$ with $\codim(X\setminus U)>c$. 
\end{prop}
Note that the derived categories appearing on the right hand side have homological dimension $\dim(X)$ with respect to the standard t-structure. Surprisingly, the limit has a smaller homological dimension which is one of the key ingredients for the investigation of the quotient categories.
\begin{theorem}
The homological dimension of $\Cohc(X)$ is $c$ for any smooth projective variety $X$. 
\end{theorem}
However, for $\dim(X) > c$ the quotient category is not of finite type over $k$ (cf.\ Proposition \ref{HomExt}) and possesses no Serre functor.
We will use Theorem 1.2 to classify all locally-finite numerical stability conditions on $\Der(\Cohc(X))\cong\Derc(X)$ in the case $c=1$. Moreover, we show that the space of stability conditions is as disconnected as it could be. The precise statement is the following.
\begin{theorem}
Let $X$ be an irreducible smooth projective variety of dimension $\dim(X)\ge 2$ and $\stab(\DER^b_{(1)}(X))$ the space of locally-finite numerical stability conditions on $\DER^b_{(1)}(X)$ equipped with the usual $\Glt$-action. Then, $\Glt$ acts free and any $\Glt$-orbit is a connected component of the complex manifold $\stab(\DER^b_{(1)}(X))$. The space of connected components is parametrized by the set of rays in the convex cone 
\[ C(X)=\big\{\omega\in \Num_1(X)_\mathbb{R} \mid \inf\{ \omega.D\mid D\subset X\mbox{ an effective divisor on }X \}>0 \big\}.\]
For each $\omega\in C(X)$ there is a unique stability condition in the component associated to $\mathbb{R}_{>0}\omega$ with heart $\Coh_{(1)}(X)$ and central charge $Z(E)=-\omega.\ce_1(E)+i\rk(E)$. 
\end{theorem}
Our definition of `numerical' is explained in Section 4. Moreover, we also obtain a classification of all locally-finite stability conditions not necessarily numerical. \\
As already mentioned, there is a close relationship between the quotient category and questions in birational geometry. It is not difficult to show (cf.\ Corollary \ref{birational}) that the quotient category $\DER^b_{(0)}(X)$ encodes the birational type of $X$. To be precise, any exact equivalence between these categories is up to a shift induced by a unique  birational transformation. How can we interpret the quotient categories $\Derc(X)$ in the case $c>0$? It turns out that any rational map $\psi:X\dashrightarrow Y$ such that $\codim(\overline{\psi^{-1}(Z)})< c$ for any closed $Z\subset Y$ with $\codim(Z)<c$ induces an exact functor $\psi^\ast:\Derc(Y)\longrightarrow \Derc(X)$. In particular, any birational transformation which is an isomorphism in codimension $c$ gives rise to an equivalence of the quotient categories. This leads directly to the following question. Can we use the quotient category $\Derc(X)$ to determine $X$ up to this stronger version of birational equivalence? In addition to the case $c=0$ this is only possible for $c=1$. The key observation to prove this is the following decomposition theorem.
\begin{theorem} 
Let $X$ and $Y$ be two irreducible smooth projective varieties of dimension at least two. Any exact $k$-equivalence $\Psi:\DER^b_{(1)}(X)\longrightarrow \DER^b_{(1)}(Y)$ has a unique decomposition $\Psi=[n]\circ L\circ \psi^\ast$ by a shift functor, a tensor product with a line bundle $L\in \Pic(X)$ and a pullback functor induced by a birational map $\psi:Y \dashrightarrow X$ which is an isomorphism in codimension one.
\end{theorem}
This leads directly to the following two results.
\begin{cor}
Two irreducible smooth projective varieties $X$ and $Y$ are isomorphic in codimension one if and only if their quotient categories $\DER^b_{(1)}(X)$ and $\DER^b_{(1)}(Y)$ are equivalent as $k$-linear triangulated categories.
\end{cor}
\begin{cor} 
Two irreducible smooth projective surfaces $X$ and $Y$ are isomorphic if and only if there is an exact $k$-equivalence between their quotient categories $\DER^b_{(1)}(X)$ and $\DER^b_{(1)}(Y)$. 
\end{cor}
The last Corollary is the dimension two version of the well know fact that two irreducible smooth projective curves $X$ and $Y$ are isomorphic if there are birational equivalent, i.e.\ if and only if there is an exact $k$-linear equivalence between their quotient categories $\DER^b_{(0)}(X)$ and $\DER^b_{(0)}(Y)$. It would be interesting to know whether an analogue statement holds in higher dimensions.\\
\\
\textbf{Question.} \it Are two irreducible smooth projective varieties $X$ and $Y$ of dimension $d$ isomorphic if and only if their quotient categories $\DER^b_{(d-1)}(X)$ and $\DER^b_{(d-1)}(Y)$ are equivalent as $k$-linear triangulated categories? More general, are $X$ and $Y$ isomorphic in codimension $c<d$ if and only if $\Cohc(X)$ and $\Cohc(Y)$ are derived equivalent?\rm\\
\\
\\
\\
\textbf{Acknowledgements.} We wish to thank D. Huybrechts for useful discussions as well as the following institutions for their support: Bonn International Graduate School in Mathematics, Imperial College London and SFB/TR 45.

\newpage

\section{Slicings and stability conditions}

We will start by recalling the definitions of a slicing and a stability condition as introduced by T. Bridgeland in \cite{Bridgeland02}. We fix an algebraically closed field $k$ and denote by $\altD$ a $k$-linear triangulated category. 
\begin{defin}[\cite{Bridgeland02}, Definition 3.3] A slicing on $\altD$ is a family $\altP=(\altP(\phi))_{\phi\in \mathbb{R}}$ of full additive subcategories of $\altD$ satisfying the following three axioms:
\begin{enumerate}
 \item $\altP(\phi+1)=\altP(\phi)[1]$ for all $\phi\in \mathbb{R}$,
 \item if $\phi_1>\phi_2$ then $\Hom_\altD(A_1,A_2)=0$ for all $A_1\in \altP(\phi_1)$ and all $A_2\in \altP(\phi_2)$,
 \item for $0\neq E\in \altD$ there is a Harder--Narasimhan filtration, i.e.\  a finite sequence of real numbers $\phi_1 > \phi_2 > \ldots > \phi_n$ and a collection of triangles
\[ \xymatrix @C=0.3cm @R=0.7cm { 0=E_0 \ar[rr] & & E_1 \ar[rr] \ar[dl] & & E_2 \ar[rr] \ar[dl] & & \ldots \ar[rr] & & E_{n-1} \ar[rr] & & E_n=E \ar[dl] \\
 & A_1 \ar@{-->}[ul] & & A_2 \ar@{-->}[ul] & & & & & & A_n \ar@{-->}[ul] } \]
with $A_j\in \altP(\phi_j)$ for all $j$. 
\end{enumerate}
A nonzero object of $\altP(\phi)$ is called semistable of phase $\phi\in \mathbb{R}$.
\end{defin}
Every bounded t-structure on $\altD$ defines a slicing. Indeed, if we denote the heart of the t-structure by $\altA$, the family $\altP(\phi)\mathrel{\mathop{:}}=\altA[\phi-1]$ for $\phi\in \mathbb{Z}$ and $\altP(\phi)=0$ for $\phi\in \mathbb{R}\setminus \mathbb{Z}$ satisfies the upper axioms. The Harder--Narasimhan filtration is the usual cohomology filtration with respect to the heart $\altA$. There is an inverse of this construction which we will explain below. Other more continuous examples of slicings are provided by stability conditions which we will introduce now.
\begin{defin}[\cite{Bridgeland02}, Definition 1.1]
A stability condition $(Z,\altP)$ on a triangulated category $\altD$ consists of a linear map $Z:\Ka(\altD)\rightarrow \mathbb{C}$, called the central charge, and a slicing $\altP$ such that for any semistable object $E$ of phase $\phi$ there is some $m(E)\in \mathbb{R}_{>0}$ with $Z(E)=m(E)\exp(i\pi\phi)$.
\end{defin}
For any interval $I\subseteq \mathbb{R}$, define  $\altP(I)$ to be the extension-closed full subcategory of $\altD$ generated by the subcategories $\altP(\phi)$ for $\phi\in I$. Bridgeland has shown that the categories $\altP(I)$ are quasi-abelian for every interval $I\subset \mathbb{R}$ of length $<1$ (\cite{Bridgeland02}, Lemma 4.3). A quasi-abelian category is an additive category with kernels and cokernels such that every pullback of a strict epimorphism is a strict epimorphism, and every pushout of a strict monomorphism is a strict monomorphism. In contrast to an abelian category, the image of a morphism is not necessarily isomorphic to its coimage. Morphisms with this additional property are called strict. Subobjects with a strict embedding are called strict and similarly for quotients. It can be shown that the additive subcategories $\altP(\phi)$ and $\altP((\phi,\phi+1])$ as well as $\altP([\phi,\phi+1))$ are always abelian for every $\phi\in \mathbb{R}$. Simple objects of $\altP(\phi)$ are called stable objects of phase $\phi$. Moreover, the pair $(\altD^{\le 0}, \altD^{\ge 0}):=(\altP((0,\infty)), \altP((-\infty,1])$ is a bounded t-structure on $\altD$ with heart $\altA:=\altP((0,1])$. Furthermore, the linear map $Z:\Ka(\altA)=\Ka(\altD) \rightarrow \mathbb{C}$ satisfies
\begin{itemize} \it
\item[(i)]  if $0\neq E\in \altA$ then $Z(E)\in H=\{r\exp(i\pi\phi)\mid r>0, 0<\phi\le 1\}\subseteq \mathbb{C}$, 
\item[(ii)]  for $0\neq E\in \altA$ there is a Harder--Narasimhan filtration, i.e.\  a finite chain of subobjects $0=E_0 \subset E_1 \subset \ldots \subset E_{n-1}\subset E_n=E$ whose factors $F_j=E_j/E_{j-1}$ are semistable objects of $\altA$ with $\phi(F_1) > \phi(F_2) > \ldots > \phi(F_n)$. An object $F\in \altA$ is said to be semistable (with respect to $Z$) if $\phi(G)\le \phi(F)$ for every subobject $0\neq G\subset F$. 
\end{itemize}
Giving a stability condition on a triangulated category $\altD$ is equivalent to giving a bounded t-structure $(\altD^{\le 0}, \altD^{\ge 0})$ on $\altD$ with heart $\altA:=\altD^{\le 0}\cap\altD^{\ge 0}$ and a linear map $Z:\Ka(\altA) \rightarrow \mathbb{C}$ satisfying the two properties (i) and (ii) (\cite{Bridgeland02}, Proposition 5.3). There is a very useful criterion to check the Harder--Narasimhan property (ii)  (see \cite{Bridgeland02}, Proposition 2.4).
\begin{itemize} \it
\item[(ii.1)] There are no infinite sequences $\ldots \subset E_{j+1} \subset E_j \subset \ldots \subset E_0$ of subobjects in $\altA$ with $\phi(E_{j+1})>\phi(E_j)$ for all $j$, and
\item[(ii.2)] there are no infinite sequences $E^0 \sur \ldots  \sur E^{j} \sur E^{j+1} \sur \ldots $ of quotients in $\altA$ with $\phi(E^{j})>\phi(E^{j+1})$ for all $j$.
\end{itemize}
The following technical property is very important to control deformations of stability conditions. 
\begin{defin}
A slicing $\altP$ is called locally-finite if there exists a real number $\eta>0$ such that for all $\phi\in\mathbb{R}$ the quasi-abelian category $\altP((\phi-\eta,\phi+\eta))\subseteq \altD$ is of finite length. A stability condition $(Z,\altP)$ is called locally-finite if the underlying slicing has this property.
\end{defin}
Note that a quasi-abelian category is called of finite length if it is artinian and noetherian, i.e.\  any descending sequence and any ascending sequence of strict subobjects becomes stationary. This is equivalent to the fact that every object of the category has a finite filtration by strict subobjects, the so-called Jordan--H\"older filtration, such that the successive quotients are simple, i.e.\  they have no strict subobjects. Since $\altP(\phi)$ is an abelian subcategory of $\altP((\phi-\eta,\phi+\eta))$, every descending respectively ascending sequence of subobjects becomes stationary. Thus, $\altP(\phi)$ is of finite length for every $\phi\in \mathbb{R}$. In particular, every semistable object has a Jordan--H\"older filtration by stable objects of the same phase if the slicing is locally-finite. \\

There is a natural topology on the set of all locally-finite stability conditions on $\altD$ (\cite{Bridgeland02}, Section 6). In order to get a finite-dimensional space of stability conditions, we fix a projection $\Ka(\altD)\sur \altN(\altD)$ onto a free abelian group of finite rank and restrict ourselves to locally-finite stability conditions whose central charge factorizes over the projection $\Ka(\altD)\sur \altN(\altD)$. We call a stability condition with this property numerical and we denote by $\stab(\altD)$ the topological subspace of all locally-finite numerical stability conditions on $\altD$. In the following all stability conditions on $\altD$ are assumed to be numerical and locally-finite.
\begin{theorem}[\cite{Bridgeland02}, Corollary 1.3] \label{bigtheorem}
For each connected component $\Sigma\subset \stab(\altD)$ there is a complex linear subspace 
\[V(\Sigma)\subset \Hom_\mathbb{Z}(\altN(\altD),\mathbb{C})\subset \Hom_\mathbb{Z}(\Ka(\altD),\mathbb{C})\] and a local homeomorphism $\pi:\Sigma \rightarrow V(\Sigma)$ which maps a stability condition to  its central charge. In particular, $\stab(X)$ has a natural structure of a finite-dimensional complex manifold such that $\pi$ is locally biholomorphic. 
\end{theorem}
Giving a stability condition $\sigma\in \Sigma$, the space $V(\Sigma)$ is characterized by 
\[ V(\Sigma)=\{U\in \Hom_\mathbb{Z}(\altN(\altD),\mathbb{C})\mid  \|U\|_\sigma < \infty\}, \] 
where
\[ \| U \|_\sigma \mathrel{\mathop{:}}= \sup\left\{ \left. \frac{|U(E)|}{|Z(E)|} \;\right| E \mbox{ semistable in } \sigma\right\}, \]
and $\| \cdot\|_\sigma$ is a norm on $V(\Sigma)$ which can be used to define the (standard) topology on $V(\Sigma)$ \cite{Bridgeland02}.\\ 
If $\altD$ is the bounded derived category $\Der(X)$ of coherent sheaves on an irreducible smooth projective variety $X$, we define $\altN(\Der(X))=\mathrel{\mathop{:}}\altN(X)$ to be the image of $\Ka(\Der(X))=\Ka(X)$ under the Mukai map $v:\Ka(X) \longrightarrow \Ho^\ast(X,\mathbb{Q})$ which associates to every element $e\in \Ka(\Der(X))$ its Mukai vector $v(e)=\ch(e)\sqrt{\td(X)}$. The set of all locally-finite numerical stability conditions on $\Der(X)$ is denoted by $\stab(X)$.\\

The group of $k$-linear exact autoequivalences $\Aut(\altD)$ of $\altD$ and the universal covering group of $\Gl2$ act continuously on $\stab(\altD)$ (\cite{Bridgeland02}, Lemma 8.2). An autoequivalence $\Psi$ acts on $\stab(\altD)$ from the left by $\Psi\cdot(Z,\altP)=(Z\circ \psi^{-1}, \altP')$, where $\psi$ is the induced action on the K-group $\Ka(\altD)$ and $\altP'(\phi)=\Psi(\altP(\phi))$. Furthermore, there is a natural action of the universal covering group $\Glt$ of $\Gl2$ from the right. The group $\Glt$ can be thought of as the set of pairs $(g,f)$, where $f:\mathbb{R}\rightarrow \mathbb{R}$ is an increasing map with $f(\phi+1)=f(\phi)+1$, and $g\in \Gl2$ is an orientation-preserving linear isomorphism on $\mathbb{R}^2$ such that the induced maps on $S^1=\mathbb{R}/2\mathbb{Z}=\mathbb{R}^2\setminus\{0\}/\mathbb{R}_{>0}$ are the same. Then $(Z,\altP)\cdot(g,f)=(Z',\altP')$, where $Z'=g^{-1}\circ Z$ and $\altP'(\phi)=\altP(f(\phi))$. The group $\Gl2$ acts in a similar way on $\Hom_\mathbb{Z}(\Ka(\altD),\mathbb{C})$ and $\pi:\stab(\altD) \rightarrow \Hom_\mathbb{Z}(\Ka(\altD),\mathbb{C})$ intertwines both actions. The action of $\Aut(\altD)$ on $\stab(\altD)$ commutes with the one of $\Glt$.\\

We will close this section by discussing the historical example of stability introduced by Mumford in \cite{MumfordGIT}. For this we consider the bounded derived category of a smooth projective curve $X$ of genus $g$ with its standard t-structure. It is left to the reader to check that $Z(E)=-\deg(E)+i\rk(E)$ satisfies the properties (i) and (ii) and defines, thus, a central charge on $\Coh(X)$. Finally, we obtain a stability condition $\sigma_{(0)}$. The (semi)stable sheaves of phase $\phi\in (0,1)$ coincide with the (semi)stable vector bundles of slope $-\cot(\phi\pi)$ in the sense of Mumford. E. Macr\`{i} \cite{Macri} has shown  that for $g\ge 1$ the group $\Glt$ acts free and transitive on $\stab(X)$, i.e.\ 
\begin{equation} \label{curves} \stab(X)=\sigma_{(0)}\cdot\Glt\,. \end{equation}
Mumford's notion of stability has been generalized in a similar way to sheaves on smooth projective varieties $X$ of any dimension $d$ and is know as $\mu$-stability. Unfortunately, the function $Z(E)=-\deg(E)+i\rk(E)$ is not a central charge on $\Coh(X)$ for $d\ge 2$ due to the existence of torsion sheaves of degree zero, i.e.\ those whose support has codimension $c\ge 2$. The vanishing of $Z$ on these sheaves violates condition (i) of a central charge. There is an easy way to overcome this difficulty. Instead of taking $\Coh(X)$, one should concider the quotient category of $\Coh(X)$ modulo torsion sheaves supported in codimension $c\ge 2$. This should motivate the next section, where we study this quotient category in detail.
\setcounter{claim3}{\value{section}}

\newpage

\section{Quotients of the derived category}

Let $X$ be an irreducible smooth projective variety of dimension $\dim(X)=d$ over an algebraically closed field $k$. For every $0\le c\le d$ we will consider the quotient of the abelian category of coherent sheaves on $X$ by the full subcategory of sheaves $E$ with $\codim\supp(E)> c$. We aim to show that this quotient category has homological dimension $c$. For further applications we compute the $K$-group of this category and some $\Ext$-groups in the case $c=1$.

\subsection{The different quotient categories}
For this subsection we allow $X$ to be quasi-projective and introduce several full subcategories of $\Coh(X)$ and $\Der(X)=\Der(\Coh(X))$.
\begin{defin}
For any natural number $0\le c \le d+1$ we define  $\Ch{c}(X)$ to be the full subcategory of $\Coh(X)$ consisting of sheaves whose support has codimension $\ge c$. Similarly, we denote by $\Dr{c}(X)$  the full subcategory of $\Der(X)$ consisting of complexes whose support has codimension $\ge c$. Note that the support of a complex is the union of the supports of its cohomology sheaves. 
\end{defin}
In particular, $\Ch{0}(X)=\Coh(X)$ and $\Ch{1}(X)$ is the the category of coherent torsion sheaves, whereas $\Ch{2}(X)$ is the category of torsion sheaves $T$ with $\codim \supp(T)\ge 2$. The category $\Ch{d+1}(X)$ is the full subcategory of $\Coh(X)$ consisting of the zero sheaf. For technical reasons we need to generalize the definitions to the case of $\ho_X$-modules and unbounded complexes in this section. We obtain the categories $\QCh{c}(X)$ and, for example, $\DER^{+,c}(\QCoh(X))$ by requiring that the closure of the support has at least codimension $c$. \\
The following lemma is an easy consequence of the fact that for every short exact sequence
\[0\longrightarrow E' \longrightarrow E \longrightarrow E'' \longrightarrow 0\]
 of $\ho_X$-modules on $X$ the supports satisfy $\supp(E) = \supp(E')\cup \supp(E'')$. 
\begin{lemma}
The categories $\Dr{c}(X)$ resp.\ $\DER^{+,c}(\QCoh(X))$ are thick triangulated subcategories of $\Der(X)$ resp.\ $\DER^+(\QCoh(X))$. Similarly, the categories $\Ch{c}(X)$ resp. $\QCh{c}(X)$ are Serre subcategories of $\Coh(X)$ resp.\ $\QCoh(X)$. 
\end{lemma}
Recall that a full subcategory $\altD'$ of a triangulated category $\altD$ is called thick if $E'\oplus E''\in \altD'$ implies $E',E''\in \altD'$. A full subcategory $\altA'$ of an abelian category $\altA$ is a Serre subcategory if for every short exact sequence 
\[ 0\longrightarrow E' \longrightarrow E \longrightarrow E'' \longrightarrow 0\]
in $\altA$ the object $E$ belongs to $\altA'$ if and only if $E'$ and $E''$ do so. In particular, $\altA'$ is an abelian subcategory.\\
The next propositions are special cases of results of Serre respectively Verdier.
\begin{prop}[\cite{Neeman}, Lemma A.2.3] \label{abelianquotient}
For every $0 \le c \le d$ there is an abelian category $\Cohc(X)$ and an exact functor $P:\Coh(X) \longrightarrow \Cohc(X)$ whose kernel is the subcategory $\Ch{c+1}(X)$ and which is universal among all exact functors $\tilde{P}:\Coh(X) \longrightarrow \altA$ between abelian categories vanishing on $\Ch{c+1}(X)$. The category $\Cohc(X)$ is called the quotient category of $\Coh(X)$ by $\Ch{c+1}(X)$. The objects of $\Cohc(X)$ are those of $\Coh(X)$, and a morphism between two objects $E$ and $F$ is represented by a roof
\[ \xymatrix @C=0.5cm @R=0.5cm { & E' \ar[dl]_s \ar[dr]^f & \\ E & & F\;, } \]
where $s$ and $f$ are morphisms in $\Coh(X)$ with $\ker(s), \coker(s) \in \Ch{c+1}(X)$.
\end{prop}
\begin{prop}[\cite{Neeman}, Theorem 2.1.8]
For every $0 \le c \le d$ there is a triangulated category $\Derc(X)$ and an exact functor $Q:\Der(X) \longrightarrow \Derc(X)$ whose kernel is the subcategory $\Dr{c+1}(X)$ and which is universal among all exact functors $\tilde{Q}:\Der(X) \longrightarrow \altT$ between triangulated categories vanishing on $\Dr{c+1}(X)$. The category $\Derc(X)$ is called the quotient category of $\Der(X)$ by $\Dr{c+1}(X)$. The objects of $\Derc(X)$ are those of $\Der(X)$, and a morphism between two objects $E$ and $F$ is represented by a roof
\[ \xymatrix @C=0.5cm @R=0.5cm { & E' \ar[dl]_s \ar[dr]^f & \\ E & & F\;, } \]
where $s$ and $f$ are morphisms in $\Der(X)$ with $C(s)\in \Dr{c+1}(X)$.
\end{prop}
For the quotient $\Derc(X)$ as well as for $\Cohc(X)$ two roofs 
\[ \xymatrix @C=0.5cm @R=0.5cm { & E' \ar[dl]_s \ar[dr]^f & \\ E & & F } \qquad\mbox {and }\qquad \xymatrix @C=0.5cm @R=0.5cm { & E'' \ar[dl]_t \ar[dr]^g & \\ E & & F } \]
represent the same morphism if there is commutative diagram
\[ \xymatrix @C=0.5cm @R=0.5cm { & & E''' \ar[dl]_u \ar[dr]^v & & \\ & E' \ar[dl]_s \ar[drrr]_f & & E'' \ar[dlll]^t \ar[dr]^g & \\ E & & & & F } \]
with $C(su)\in \Dr{c+1}(X)$ respectively $\ker(su),\coker(su)\in \Ch{c+1}(X)$. \\
Similar quotients exist in the case of $\ho_X$-modules and we denote the quotient categories by $\QCohc(X)$ and $\QDerc$.\\
The inclusion functors $I:\Coh(X)\hookrightarrow \QCoh(X)$ and $\Der(X) \hookrightarrow \DER^+(\QCoh(X))$ induce natural exact functors $I_{(c)}:\Cohc(X)\longrightarrow \QCohc(X)$ and $\Derc(X)\longrightarrow \QDerc$. Furthermore, the functor $\Der(P):\Der(X)\longrightarrow \Der(\Cohc(X))$ factorizes over the quotient functor $Q:\Der(X)\longrightarrow \Derc(X)$ and similarly in the case of $\ho_X$-modules. Indeed, $\Der(P)$ commutes with the cohomology 
functors
\[ \xymatrix @C=3cm @R=1cm { {\Der(X)} \ar[r]^{\Der(P)} \ar[d]_{H^i} & \Der(\Cohc(X)) \ar[d]^{H^i} \\ \Coh(X) \ar[r]^P & \Cohc(X). } \]
For $E\in \Dr{c+1}(X)$ we have $H^i (\Der(P)(E))=P( H^i(E))=0$ for all $i\in \mathbb{Z}$ and, therefore, $\Der(P)(E)=0$. The existence of the functor $T:\Derc(X) \longrightarrow \Der(\Cohc(X))$ follows by the universal property of $Q:\Der(X)\longrightarrow \Derc(X)$. We can summarize all functors in the following commutative diagram
\[ \xymatrix @C=1cm @R=1cm { {\Derc(X)} \ar[rrr] \ar[dd]_T & & & {\QDerc} \ar[dd] \\ & {\Der(X)} \ar[ul]_Q \ar[dl]_{\Der(P)} \ar[r]^(0.4){\Der(I)} & {\DER^{+}(\QCoh(X))} \ar[ur] \ar[dr] & \\ {\Der(\Cohc(X))} \ar[rrr]^{\Der(I_{(c)})} & & & {\DER^{+}(\QCohc(X)).} } \]
\begin{prop} \label{equivcat}
The functors $\Der(I)$ and $\Der(I_{(c)})$ are equivalences of $\Der(X)$ resp.\ $\Der(\Cohc(X))$ with the full subcategories $\DER^b_{coh}(\QCoh(X))$ resp.\ $\DER^b_{coh}(\QCohc(X))$ of bounded complexes with  cohomology sheaves in $\Coh(X)$ resp.\ $\Cohc(X)$. Furthermore, the functor $T$ is an equivalence of categories. Hence
\[ \Derc(X) \cong \Der(\Cohc(X))\cong \DER^b_{coh}(\QCohc(X)) \subset \DER^+(\QCohc(X)).\]
\end{prop}
\begin{proof} 
The assertion for $\Der(I)$ is well known (see e.g.\ \cite{HuybFourMuk}, Corollary 3.4 and Proposition 3.5). The case of $\Der(I_{(c)})$ is proved in the same way. For this we introduce the full subcategory $\Qcohc(X)$ of $\QCohc(X)$ whose objects are quasi-coherent sheaves. Every quasi-coherent sheaf $F$ has a resolution 
\[ 0 \longrightarrow F\longrightarrow I^0 \longrightarrow I^1 \longrightarrow I^2 \longrightarrow \ldots \]
by quasi-coherent sheaves $I^i$ which are injective as $\ho_X$-modules (see \cite{Hartshorne66}, II, 7.18). As injective $\ho_X$-modules remain injective in $\QCohc(X)$ (see the end of the proof of the next Lemma), we obtain an equivalence 
\[ \Der(\Qcohc(X)) \cong \DER^b_{qcoh}(\QCohc(X)) \subset  \DER^+(\QCohc(X))\]
of $\Der(\Qcohc(X))$ with the full subcategory $\DER^b_{qcoh}(\QCohc(X))$ of bounded complexes with quasi-coherent cohomology sheaves (cf.\ \cite{HuybFourMuk}, Proposition 2.42). It remains to show that the category $\Der(\Cohc(X))$ is equivalent to the full subcategory $\DER^b_{coh}(\Qcohc(X))$ of $\Der(\Qcohc(X))$ consisting of bounded complexes with coherent cohomology sheaves. The functor is induced by the inclusion $\Cohc(X)\hookrightarrow \Qcohc(X)$. The proof of Proposition 3.5 in \cite{HuybFourMuk} applies literally to this case. Indeed, for every epimorphism $f:G\sur F$ in $\Qcohc(X)$ with $F\in \Cohc(X)$ there is a coherent subsheaf $G'\subseteq G$ such that the restriction $f:G' \longrightarrow F$ remains an epimorphism. For this, we represent $f$ by a roof
\[ \roof{G}{E}{F}{s}{\tilde{f}} \]
with $\ker(s),\coker(s),\coker(\tilde{f})\in \QCh{c+1}(X)$. Thus, $\tilde{f}$ defines a surjective sheaf homomorphism $\hat{f}:G|_U \sur F|_U$ on some open subset $U\subset X$ with $\codim(X\setminus U)>c$. Let $\hat{G}\subset G|_U$ be a coherent subsheaf of $G|_U$ such that $\hat{f}:\hat{G} \rightarrow F|_U$ is still surjective (cf.\ \cite{HuybFourMuk}, Prop. 3.5). Denote by $G'\subset G$ a coherent subsheaf extending $\hat{G}$ (cf. \cite{Hartshorne}, II, Exec. 5.15).  The roof
\[ \roof{G'}{s^{-1}(G')}{F}{s}{\tilde{f}}\]
represents the required restriction. The rest of the argument is the same as in \cite{HuybFourMuk}. \\
The proof that $T$ is an equivalence is more involved. As the assertion is not used in the sequel, we will skip the proof. The interested reader will find it in Appendix A.
\end{proof}
For two $\ho_X$-modules $F$ and $G$ we introduce the following $k$-vector space
\[ \HOM_{(c)}(F,G)\mathrel{\mathop{:}}=\varinjlim_{\codim(X\setminus U)>c} \Hom(F|_U,G|_U), \]
where the limes is taken over the directed set of open subsets $U\subset X$ with $\codim(X\setminus U)>c$. We leave it to the reader to check that $\HOM_{(c)}(-,-)$ is indeed a functor from $\QCoh(X)^{op}\times \QCoh(X)$ into the category of $k$-vector spaces. Moreover, it is left exact in the second variable. Thus, we can construct for a fixed $\ho_X$-module $F$ the right derived functors 
\[ R^i\HOM_{(c)}(F,-)=\mathrel{\mathop{:}}\EXT^i_{(c)}(F,-)\] 
of the left exact functor $\HOM_{(c)}(F,-)$ on $\DER^+(\QCoh(X))$ since $\QCoh(X)$ has enough injective objects. Furthermore, every morphism $F\longrightarrow F'$ induces functor morphisms $\EXT^i_{(c)}(F',-) \longrightarrow \EXT^i_{(c)}(F,-)$ compatible with the long exact sequences associated to a short exact sequence $0\longrightarrow G' \longrightarrow G \longrightarrow G'' \longrightarrow 0$. The following lemma shows the importance of these functors.
\begin{lemma}
For $F',G'\in \QCohc(X)$ and $i\in \mathbb{N}$ we introduce the following shorthand
\[ \Ext^i_{(c)}(F',G')\mathrel{\mathop{:}}=\Hom_{\DER^+(\QCohc(X))}(F',G'[i]). \]
Then, there are functor isomorphisms $\EXT^i_{(c)}(F,-)\cong \Ext^i_{(c)}(P(F),P(-))$, natural in $F\in \QCoh(X)$, where $P:\QCoh(X)\longrightarrow \QCohc(X)$ is the quotient functor which maps $G\in \QCoh(X)$ to $G\in \QCohc(X)$. 
\end{lemma}
\begin{proof}
Let us start by considering the case $i=0$. An element of $\HOM_{(c)}(F,G)$ is represented by a sheaf homomorphism $f:F|_U\longrightarrow G|_U$ with $U\subset X$ open and $\codim(X\setminus U)>c$. Take a sheaf $E\subset F\oplus G$ on $X$ with $E|_U=\Gamma_f$, where $\Gamma_f\subset F|_U\oplus G|_U$ denotes the graph of $f$.\footnote{There is a canonical choice given by $E=\rho^{-1}(j_\ast \Gamma_f)$ with $j_\ast:U \hookrightarrow X$ and $\rho:F\oplus G \longrightarrow j_\ast(F|_U\oplus G|_U)$.} We associate to the element represented by $f$ the homomorphism $\phi\in \Hom_{\QCohc(X)}(P(F),P(G))$ represented by the roof
\[ \roof{F}{E}{G.}{pr_1}{pr_2}\]
The morphism $\phi$ is independent of the choice of $f$ and $E$. To see this, let us consider another choice $f':F|_{U'}\longrightarrow G|_{U'}$ and $E'\subset F\oplus G$ with $E'|_{U'}=\Gamma_{f'}$. Since $f$ and $f'$ represent the same element in $\HOM_{(c)}(F,G)$, there is an open subset $V\subset U\cap U'$ with $\codim(X\setminus V)>c$ and $f|_V=f'|_V$. We take a subsheaf $E''$ of $E\cap E' \subset F\oplus G$ extending $\Gamma_f|_V=\Gamma_{f'}|_V$. With the inclusion maps $\imath:E''\hookrightarrow E$ and $\imath':E''\hookrightarrow E'$ we get the following commutative diagram
\[ \xymatrix @C=0.5cm @R=0.5cm { & & E'' \ar[dl]_\imath \ar[dr]^{\imath'} & & \\ & E \ar[dl]_{pr_1} \ar[drrr]_{pr_2} & & E' \ar[dlll]^{pr_1} \ar[dr]^{pr_2} & \\ F & & & & G } \]
which shows the equivalence of the two roofs. There is an inverse of this construction given as follows. Take some $\phi\in \Hom_{\QCohc(X)}(P(F),P(G))$ and represent it by some roof
\[ \roof{F}{E}{G.}{t}{g} \]
We associate to $\phi$ the element of $\HOM_{(c)}(F,G)$ represented by $f=g|_U\circ (t|_U)^{-1}:F|_U \longrightarrow G|_U$, where $U$ is the complement of $\supp(\ker(t))\cup \supp(\coker(t))$. It is left to the reader to show that this element is independent of the choice of the roof, that these bijections are additive and that they form a natural transformation between $\HOM_{(c)}(-,-)$ and $\Hom_{\QCohc(X)}(P(-),P(-))$. \\
As $P:\QCoh(X)\longrightarrow \QCohc(X)$ is exact, both sequences of functors are $\delta$-functors and the sequence $\EXT^i_{(c)}(F,-)$ is universal by construction. Note that for every injective $\ho_X$-module $I$ on $X$ the sheaf $I|_U$ is injective in $\QCoh(U)$ for every open subset $U\subset X$ (\cite{Hartshorne}, III, Lemma 6.1). Using the first part of the proof, we see that $P(I)$ is injective in $\QCohc(X)$ and, therefore, $\Ext^i_{(c)}(P(F),P(I))=0$ for all $i>0$. Since every $\ho_X$-module is a subsheaf of an injective $\ho_X$-module, the functors $\Ext^i_{(c)}(P(F),P(-))$ are effaceable for all $i>0$. Hence, they are universal (see \cite{Grothendieck57}, II, 2.2.1). As universal $\delta$-functors are unique up to isomorphism, the assertion follows from the first part of the proof.    
\end{proof}
Note that the functor 
\[\varinjlim_{\codim(X\setminus U)>c}\] 
is exact since the index set is directed. Using this and the exactness of the restriction functors $|_U$, a simple analysis of the construction of the derived functors $\EXT^i_{(c)}$ reveals 
\[ \EXT^i_{(c)}(F,G) \cong \varinjlim_{\codim(X\setminus U)>c} \Ext^i(F|_U,G|_U). \]
Combining this with the lemma yields the following corollary.
\begin{cor} \label{Ext}
With the shorthands of the last lemma we obtain for all $F,G\in \QCoh(X)$ and $i\in \mathbb{N}$ natural isomorphisms
\[ \Ext^i_{(c)}(F,G)  \cong  \varinjlim_{\codim(X\setminus U)>c} \Ext^i(F|_U,G|_U),\]
where we suppressed $P$ from the notion. In particular,
\[ \Hom_{(c)}(F,G) \cong \varinjlim_{\codim(X\setminus U)>c} \Hom(F|_U,G|_U).\]
\end{cor}
\newpage
\textbf{Remarks.} 
\begin{itemize} 
\item We can use the latter isomorphisms to define the categories $\QCohc(X)$ and $\Cohc(X)$ such that the homomorphism groups are given by the expression on the right hand side of the equation. Composition is defined by composing representing $\ho_X$-module homomorphisms.
\item Although the left hand side of the equations is naturally defined on the quotient category $\QCohc(X)$, the right hand side is not as $\Ext^i$ makes no sense on $\QCohc(U)$. Hence, the equations are isomorphisms of functors on $\QCoh(X)$. 
\end{itemize}
There is a nice interpretation of Corollary \ref{Ext} in terms of birational geometry. For this we define a subcategory $\Var^c$ of the category $\Var$ of smooth irreducible quasi-projective varieties over a fixed field $k$. Objects of $\Var^c$ are smooth irreducible varieties and a morphism $f:X\longrightarrow Y$ belongs to $\Var^c$ if  $\codim(f^{-1}(Z))>c$ for any  closed subset $Z\subset Y$ with $\codim(Z)>c$. Note that this is an empty condition for $c=0$ or $c>\dim(X)$. Let us denote by $O^c$ the class of open embeddings $U\hookrightarrow X$ such that $\codim(X\setminus U)>c$. The class of morphisms $O^c$ is localizing in $\Var^c$ and we denote by $\Var_{(c)}$ the localization of $\Var^c$ with respect to $O^c$. Thus, morphisms in $\Var_{(c)}$ are represented by roofs
\[ \roof{X}{U}{Y,}{i}{f}\]
where $i:U\hookrightarrow X$ is an open embedding with $\codim(X\setminus U)>c$ and $f$ is a morphism in $\Var^c$. In other words, morphisms in $\Var_{(c)}$ are rational maps $f:X\dashrightarrow Y$ defined in codimension $c$ such that $\codim(\overline{f^{-1}(Z)})>c$ for any closed $Z\subset Y$ with $\codim(Z)>c$. In particular, $\Var_{(0)}$ is the category of birational maps. Moreover, the full category $\Var_c\subset \Var$ of smooth varieties $X$ of dimension $\dim(X)\le c$ is also a full subcategory of $\Var_{(c)}$, i.e.\ 
\[ \Var_c \subset \Var_{(c)} \subset \Var_{(0)}.\]
Two varieties $X$ and $Y$ are isomorphic in $\Var_{(c)}$ if there is a birational map $f:X\dashrightarrow Y$ which is an isomorphism in codimension $c$. Note that if $\dim(X)=\dim(Y)\le c+1$ this implies $X\cong Y$ as varieties.\\
With respect to the (derived) pullback, the bounded derived category $\Der(-)$ is a contravariant functor on $\Var$ with values in the category of essentially small triangulated categories. This functor does not descent to a functor on $\Var_{(c)}$, but there are two natural ways to handle this problem. The first approach uses the quotient categories $\Derc(-)$ to define a contravariant functor on $\Var_{(c)}$. Given a morphism $f:X\dashrightarrow Y$ represented by a roof as above, we define $f^\ast:\Derc(Y) \longrightarrow \Derc(X)$ as follows. For any complex $E\in \Derc(Y)$ the usual derived pullback $Lf^\ast(E)$ is a complex of coherent sheaves on the open subset $U\subset X$. The direct image $i_\ast( Lf^\ast(E))$ is a complex of quasi-coherent sheaves on $X$. By standard arguments this complex contains a subcomplex $F\subset i_\ast( Lf^\ast(E))$ of coherent sheaves such that $F|_U=Lf^\ast(E)$. Considered as an object of $\Derc(X)$ this complex $F$ is (up to isomorphism) independent of the chosen roof representing the rational map $f:X \dashrightarrow Y$ and the choice of the subcomplex $F$. Using this, one can construct a functor $f^\ast:\Derc(Y) \longrightarrow \Derc(X)$ with $F=f^\ast(E)$.\\ 
Another way to define a contravariant functor on $\Var_{(c)}$ into the category of essentially small triangulated categories is to consider the direct limit category
\[ \altD_{(c)}(X)\mathrel{\mathop{:}}= \varinjlim_{\codim(X\setminus U)>c} \Der(U),\]
where the limit is taken over the directed set of open subsets $U$ of $X$ with $\codim(X\setminus U)>c$. Here we use the natural restriction functors $i^\ast:\Der(U)\longrightarrow \Der(V)$ for $i:V\hookrightarrow U$ to form a direct system of triangulated categories. It is easy to see that $\altD_{(c)}(-)$ defines a covariant functor on $\Var_{(c)}$. There is a natural transformation $\varepsilon:\Derc(-)\longrightarrow \altD_{(c)}(-)$ induced by the natural functor $\Der(X)\longrightarrow \altD_{(c)}(X)$. 
\begin{prop}
The natural transformation $\varepsilon_X:\Derc(X)\longrightarrow \altD_{(c)}(X)$ is an equivalence of triangulated categories for any smooth irreducible quasi-projective variety X.
\end{prop}
\begin{proof}
We have to show that $\varepsilon_X$ induces an isomorphism 
\[\Hom_{(c)}(F,G)\cong \varinjlim_{\codim(X\setminus U)>c} \Hom_{\Der(U)}(F|_U,G|_U) \quad\mbox{for any } F,G\in \Der(X).\] 
Using the long exact $\Hom$-sequences associated to the cohomology filtration of $F$ and $G$ as well as the five-lemma, we can prove this by induction on the length of the complexes $F$ and $G$. For the initial case of length one we use Corollary \ref{Ext}. It remains to show that every object in $\altD_{(c)}(X)$ is isomorphic to some complex of coherent sheaves on $X$. Any object of $\altD_{(c)}(X)$ is represented by some complex $F$ of coherent sheaves on an open subset $i:U\hookrightarrow X$. As above, the direct image $i_\ast( F)$ contains a subcomplex $F'$ of coherent sheaves with $F'|_U\cong F$. As an object in $\Derc(X)$, $F'$ is independent (up to isomorphism) of the choice of the complex $F$ and the choice of the subcomplex in $i_\ast(F)$.
\end{proof}
Note that every bounded derived category appearing in the definition of $\altD_{(c)}(X)$ has homological dimension $\dim(X)$. We will show in the next subsection that the homological dimension of the limit category $\altD_{(c)}(X)$ is $c$ and, thus, independent of the dimension of $X$. Moreover, we will prove in Section 5 that for $c\le 1$ the triangulated categories $\altD_{(c)}(X)\cong \Derc(X)$ contain enough information to classify $X$ as an object in $\Var_{(c)}$.

\subsection{Properties of the quotient category}

Let $X$ be an irreducible smooth projective variety. This subsection contains the proofs that $\Cohc(X)$ is noetherian and has homological dimension $c$. Before going into details we will make some remarks about the $\Ka$-group of $\Cohc(X)$. If $\altA$ is a Serre subcategory of an abelian category $\altB$, we can construct the quotient category $\altB/\altA$ and there is the following exact sequence of $\Ka$-groups (cf. \cite{Weibel}, II, Theorem 6.4)
\[ \Ka(\altA) \longrightarrow \Ka(\altB) \longrightarrow \Ka(\altB/\altA) \longrightarrow 0.\]
The images $F^c\Ka(X)$ of $\Ka(\Ch{c}(X))$ in $\Ka(X)=\Ka(\Coh(X))$ form a descending filtration, called the topological filtration. Hence, 
\[ \Ka(\Cohc(X))= \Ka(X)/F^{c+1}\Ka(X).\] 
The Chern character induces a natural isomorphism 
\[ \Ka(\Cohc(X))_{\mathbb{Q}}\cong  \bigoplus_{i\le c} A^i(X)_\mathbb{Q},\]
where $A^i(X)_\mathbb{Q}=A^i(X)\otimes \mathbb{Q}$ is the $i$-th rational Chow-group of $X$. See \cite{Murre} for more details.
\begin{prop}
For every $0\le c\le d=\dim(X)$ the abelian category $\Cohc(X)$ is noetherian. 
\end{prop}
\begin{proof}
If we fix some ample divisor $H$ on $X$, we can define the Hilbert polynomial 
\[ P(E,m)= \sum_{i=0}^d \alpha_i(E)\frac{m^i}{i!} \] 
for any coherent sheaf $E$ by $P(E,m)\mathrel{\mathop{:}}=\chi(X,E\otimes\ho_X(mH))$. The coefficients $\alpha_i(E)$ have the following properties (cf.\  \cite{HuybLehn}, Lemma 1.2.1)
\begin{enumerate}
 \item $\alpha_i(E)\in \mathbb{Z}$ for all $0\le i\le d$,
 \item $\alpha_i(E)=0$ for all $i>\dim(E)\mathrel{\mathop{:}}=\dim \supp(E)$ and $\alpha_{\dim(E)}(E)>0$,
 \item the function $E\longmapsto \alpha_i(E)$ is additive on $\Coh(X)$ for all $0\le i\le d$,
 \item the additive function $\alpha_i(-)$ descends to an additive function  on $\Cohc(X)$ for all $d-c \le i \le d$.
\end{enumerate}
Note that $\dim(E'),\dim(E'')\le \dim(E)$ for every short exact sequence $0 \longrightarrow E' \longrightarrow E \longrightarrow E''\longrightarrow 0$ in $\Cohc(X)$ with $E\neq 0$. Let us denote by $\torssh(E)$ the biggest subsheaf $E'$ of the sheaf $E$ with $\dim(E')<\dim(E)$. A sheaf $E$ with $\torssh(E)=0$ is called pure. In particular, $E/\torssh(E)$ is a pure sheaf. If $U\subset X$ is an open subset with $\dim(X\setminus U)<\dim(E)$ then $E$ is pure if and only $E|_U$ is pure.\\
Let us consider a sequence $E=E_0 \sur E_1 \sur E_2 \sur \ldots$ of quotients in $\Cohc(X)$. By induction on $\dim(E)$ we show that the sequence is stationary. The assertion is trivial for $\dim(E)<d-c$. Let us consider the following commutative diagram in $\Cohc(X)$ with exact rows and columns
\begin{equation} \label{noether} \xymatrix @R=1cm @C=1cm { & \ker(f_j) \ar@{^{(}->}[d]  &  & \ker(g_j) \ar@{^{(}->}[d] & \\
 0 \ar[r] & \torssh(E_j) \ar[r] \ar[d]^{f_j} & E_j \ar[d] \ar[r] & E_j/\torssh(E_j) \ar[d]^{g_j} \ar[r] & 0 \\
0 \ar[r] & \torssh(E_{j+1}) \ar@{->>}[d] \ar[r] & E_{j+1} \ar[d] \ar[r] & E_{j+1}/\torssh(E_{j+1}) \ar@{->>}[d] \ar[r] & 0 \\
& \coker(f_j)  & 0  & \coker(g_j)  &  } 
\end{equation}
By the snake lemma, $\coker(g_j)=0$ for all $j\in \mathbb{N}$. Without loss of generality we can assume $\dim(E)=\dim(E_j)=\dim(E_j/\torssh(E_j))=\mathrel{\mathop{:}}e\ge d-c$ for all $j\in \mathbb{N}$ and we conclude 
\[ \alpha_e(E_j/\torssh(E_j))=\alpha_e(E_{j+1}/\torssh(E_{j+1})) + \alpha_e(\ker(g_j))\]
with all values of $\alpha_e$ in $\mathbb{N}$. Thus, the sequence $(\alpha_e(E_j/\torssh(E_j)))_{j\in \mathbb{N}}$ is descending and $\alpha_e(\ker(g_j))=0$ follows for all $j\gg0$. Hence $\dim(\ker(g_j))<\dim(E)=\dim(E_j/\torssh(E_j))$ for $j\gg0$. If $\ker(g_j)\neq 0$ in $\Cohc(X)$, there is an open subset $U_j\subset X$ with $\codim(X\setminus U_j)>c$ such that $\ker(g_j)|_{U_j}$ is a subsheaf of $E_j/\torssh(E_j)$ (cf.\ Corollary \ref{Ext}) of smaller dimension for $j\gg0$ which is a contradiction to the purity of $E_j/\torssh(E_j)$. Hence, $\coker(f_j)=\ker(g_j)=0$ for $j\gg 0$. By  applying  the induction hypothesis to the sequence $\torssh(E_n) \sur \torssh(E_{n+1}) \sur \ldots$ with $n\gg 0$ and the five lemma to the diagram (\ref{noether}), we conclude the assertion. 
\end{proof}
\textbf{Remark.} As in the case $\Coh(X)$ one can show that $\Cohc(X)$ is not artinian for $c>0$. The exceptional case $c=0$ will be studied in the next subsection. \\

Finally, we come to the main theorem of this section about quotient categories.
\begin{theorem}
The homological dimension of $\Cohc(X)$ is at most $c$. 
\end{theorem}
\begin{proof}
We will use Corollary \ref{Ext} to show that $\Ext^i_{(c)}(F,G)=0$ for all $F,G\in \Cohc(X)$ and all $i>c$. In order to compute $\Ext^i(F|_U,G|_U)$ in the expression
\begin{equation} \label{Ext2} \Ext^i_{(c)}(F,G) \cong  \varinjlim_{\codim(X\setminus U)>c} \Ext^i(F|_U,G|_U),\end{equation}
we use the spectral sequence
\begin{equation} \label{specseq} \Ho^p(U,\altE xt^q(F|_U,G|_U)) \Longrightarrow \Ext^{p+q}(F|_U,G|_U).\end{equation}
The crucial argument is given by the following lemma.
\begin{lemma}
Let us fix $F,G\in \Coh(X)$, an open subset $U\subset X$ with $\codim(X\setminus U)>c$ and two numbers $p,q\in \mathbb{N}$ with $p+q>c$. Then, there is an open subset $V\subset U$ with $\codim(X\setminus V)>c$ such that $\Ho^p(V,\altE xt^q(F|_V,G|_V))=0$.   
\end{lemma}
\begin{proof}
If the stalk $\altE xt^q(F,G)_x=\Ext^q_{\ho_{X,x}}(F_x,G_x)$ is not zero for $x\in X$, we conclude $\dimh(F_x)\ge q$, where $\dimh(F_x)$ is the homological dimension of $F_x$, i.e.\ the minimal length of a projective resolution of $F_x$. By the Auslander--Buchsbaum formula $\dimh(F_x)+\depth(F_x)=\dim(\ho_{X,x})$, where $\depth(F_x)\ge 0$ denotes the depth of $F_x$, we conclude $\dim(\ho_{X,x})\ge q$. Hence, $\codim\supp(\altE xt^q(F,G))\ge q$. If $q>c$, we take $V=U\setminus \supp(\altE xt^q(F,G))$ and the assertion follows. In the remaining case $q\le c$ we use the following claim to prove the lemma.\\

\textbf{Claim.} \it Let $Y$ be a projective scheme, $E$ a coherent sheaf on $Y$ and let $W\subset Y$ be an open subset with $\codim(Y\setminus W)>\kappa$. There is an open subset $W'\subset W$ with $\codim(Y\setminus W')>\kappa$ such that $\Ho^p(W',E|_{W'})=0$ for all $p>\kappa$. \rm \\

Indeed, take $Y=\supp(\altE xt^q(F,G))$, $E=\altE xt^q(F,G)$, $W=Y\cap U$ and $\kappa=c-q$ and apply the claim. The open subset $V\subset U$ is given by $U\setminus (Y \setminus W')$. \\
To prove the claim, we choose $\kappa+1$ effective divisors $D_0,\ldots,D_\kappa$ in the linear system associated to some very ample divisor $H$ on $Y$ such that $Y\setminus U\subset D_0 \cap \ldots \cap D_\kappa$ and $\codim(D_0 \cap \ldots \cap D_\kappa)>\kappa$. We can use the affine cover $(Y\setminus D_i)_{i=0}^\kappa$ of $W'\mathrel{\mathop{:}}=\bigcup_{i=0}^\kappa Y\setminus D_i$ to compute the cohomology of $E|_{W'}$ (see \cite{Hartshorne}, III, Theorem 4.5). Thus, $\Ho^p(W',E|_{W'})=0$ for all $p>\kappa$. 
\end{proof}
If $V$ in the lemma is independent of $p$ and $q$, we obtain $\Ext^i(F|_V,G|_V)=0$ for all $i>c$ by the spectral sequence (\ref{specseq}). The equation (\ref{Ext2}) would prove the theorem. Unfortunately, $V$ depends at least on $q$ and we cannot replace $V$ by the intersection of the $V_q$'s associated to the values of $q$ since $\Ho^p(V,E|_V)=0$ is not valid after replacing $V$ by an open subset of $V$. For example $\Ho^1(\mathbb{A}_k^2\setminus\{0\},\ho_{\mathbb{A}_k^2\setminus\{0\}})\neq 0$ by the local cohomology sequence associated to $\{0\}\subseteq \mathbb{A}_k^2$. The correct interpretation of the lemma is the formula
\[ \varinjlim_{\codim(X\setminus U)>c} \Ho^p(U,\altE xt^q(F|_U,G|_U)) = 0 \qquad\mbox{ for all }p+q>c.\]
Using this and equation (\ref{Ext2}), we can prove the theorem by applying the following lemma which is proved in Appendix B. 
\end{proof}
\begin{lemma} \label{limitsequence}
For every $F,G\in \QCoh(X)$ the spectral sequences 
\[\Ho^p(U,\altE xt^q(F|_U,G|_U)) \Longrightarrow \Ext^{p+q}(F|_U,G|_U)\] 
form an inductive system. There is a limit spectral sequence converging against 
\[ E^i = \varinjlim_{\codim(X\setminus U)>c} \Ext^i(F|_U,G|_U)\]
with $E_2$-term 
\[ E_2^{p,q} = \varinjlim_{\codim(X\setminus U)>c} \Ho^p(U,\altE xt^q(F|_U,G|_U)). \]
\end{lemma}
\addtocounter{prop}{1}
\setcounter{cor4}{\value{prop}}
\textbf{Remark \arabic{section}.\arabic{prop}.} It is not difficult to show that the homological dimension of $\Cohc(X)$ is exactly $c$. For $c=d$ we just mention $\Cohc(X)=\Coh(X)$. For $c<d$ take a smooth subvariety $Y\subset X$ of codimension $c$. Using a Koszul resolution of $\ho_Y$, which exists at least locally, one shows that $\altE xt^c(\ho_Y,\ho_Y)$ is a line bundle $L$ on $Y$. Now, we define $U\subset X$ to be $(X\setminus Y) \cup (X\setminus D)$, where $D$ is some effective very ample divisor on $X$ intersecting $Y$ transversely. Since $U\cap Y= Y\setminus D$ is affine, we have got infinitely many sections of $L$ on $U$, i.e. $\dim\Ho^0(U,\altE xt^c(\ho_Y,\ho_Y)|_U)=\infty$. These sections do not vanish if we restrict them to open subsets $V$ of $U$ with $\codim(X\setminus V)>c$ since $\codim_Y(Y\setminus(V\cap Y))>0$. Thus, 
\[  E_2^{0,c} = \varinjlim_{\codim(X\setminus U)>c} \Ho^0(U,\altE xt^c(\ho_Y|_U,\ho_Y|_U))\]
is an infinite-dimensional vector space which survives in the limit spectral sequence, and $\dim\Ext^c_{(c)}(\ho_Y,\ho_Y)=\infty$ follows.

\subsection{The cases $c=0$ and $c=1$}

In the last part of Section \arabic{section} we consider the cases $c=0$ and $c=1$ more carefully. The case $c=0$ is well known, but we will include it for the sake of completeness. The results for $c=1$ are important for the classification of stability conditions and autoequivalences on $\Der(\Coh_{(1)}(X))$ in the next two sections. As before, $X$ is an irreducible smooth projective variety over an algebraically closed field $k$.

\begin{prop} \label{case0}
The category $\Coh_{(0)}(X)$ is equivalent to the abelian category of finite-dimensional $K(X)$-vector spaces, where $K(X)$ is the function field of $X$. In particular, $\Ka(\Coh_{(0)})\stackrel{\rk}{=}\mathbb{Z}$. 
\end{prop}
\begin{proof} Using Corollary \ref{Ext} we get
 \[ \Hom_{(0)}(\ho_X,\ho_X)\quad=\varinjlim_{\emptyset \neq U\subset X \,\mbox{\scriptsize open}} \Hom(\ho_U,\ho_U) \quad= \quad K(X). \]
Thus, there is an additive fully faithfull functor $\Phi:\Vect^{f.d.}_{K(X)} \longrightarrow \Coh_{(0)}(X)$ from the category of finite-dimensional $K(X)$-vector spaces to $\Coh_{(0)}(X)$ mapping $K(X)$ to $\ho_X$. As every short exact sequence in $\Vect^{f.d.}_{K(X)}$ splits, $\Phi$ is an exact functor. It remains to show that every coherent sheaf $E$ is isomorphic to $\ho_X^{\oplus \rk(E)}$ in $\Coh_{(0)}$. By Serre's theorem there is a short exact sequence
\[ 0\longrightarrow \ho_X(-mH)^{\oplus\rk(E)} \longrightarrow E \longrightarrow T \longrightarrow 0\]
in $\Coh(X)$ for some effective ample divisor $H$ on $X$, some $m\in \mathbb{N}$ and  some torsion sheaf $T$. Hence, $E\cong \ho_X(-mH)^{\oplus\rk(E)}$ in $\Coh_{(0)}$. Using the exact sequence
\[ 0\longrightarrow \ho_X(-mH) \longrightarrow \ho_X \longrightarrow \ho_{mH}\longrightarrow 0\]
we obtain $\ho_X(-mH)\cong \ho_X$ in $\Coh_{(0)}(X)$ and the assertion follows.
\end{proof}
\begin{cor} \label{birational}
Two irreducible smooth projective varieties are birational if and only if their quotient categories $\DER^b_{(0)}(X)$ and $\DER^b_{(0)}(Y)$ are equivalent.
\end{cor}
\begin{proof}
Note that $\DER^b_{(0)}(-)$ is a covariant functor on the category $\Var_{(0)}$ of rational maps (see Subsection \arabic{section}.1). On the other hand, every equivalence $\Psi:\DER^b_{(0)}(X) \longrightarrow \DER^b_{(0)}(Y)$ has (up to isomorphism) a unique decomposition $[n]\circ \psi^\ast$, where $\psi:Y\dashrightarrow X$ is birational map. Indeed, every object in $\DER^b_{(0)}(Y)$ is the direct sum of the indecomposable objects $\ho_Y[n]$. Thus, the image of the indecomposable object $\ho_X$ is (isomorphic to) $\ho_Y[n]$ for some integer $n\in \mathbb{Z}$. It follows, that $[-n]\circ \Psi$ is determined by the isomorphism $K(X)\cong \Hom_{(0)}(\ho_X,\ho_X) \longrightarrow \Hom_{(0)}(\ho_Y,\ho_Y)\cong K(Y)$ given by a birational map $\psi:Y \dashrightarrow X$.
\end{proof}
We will generalize this result to the case $c=1$ and discuss the cases $c>1$ in Section 5. In the remaining part of this subsection we prove some results necessary for later sections. Let us start with the computation of the $\Ka$-group.\\
To compute the $\Ka$-group of $\Coh_{(1)}(X)$ one can use the remarks at the beginning of Subsection \arabic{section}.2. The result is well know (cf.\ \cite{Manin1},\cite{Murre}), but we will include the proof for completeness.
\begin{prop}   The $\Ka$-group of the abelian category $\Coh_{(1)}$ is 
 \[ \Ka(\Coh_{(1)}(X))\cong \mathbb{Z} \oplus\Pic(X),\]
and the isomorphism is given by the additive function $\rk\oplus \det$.
\end{prop}
\setcounter{claim1}{\value{prop}}
\begin{proof}
If we associate to a Weil divisor $D=\sum_{i=1}^p n_iD_i$ with irreducible components $D_i$ the class $\sum_{i=1}^p n_i \cl\ho_{D_i}$ in the $\Ka$-group $\Ka(\Coh_{(1)}(X))$, we obtain a group homomorphism $\tilde{\Psi}$ from the Weil group $\Weil(X)$ of $X$ into this $\Ka$-group. The short exact sequence 
\[ 0\longrightarrow \ho_X \longrightarrow \ho_X(D) \longrightarrow \ho_X(D)|_D\cong \ho_D \longrightarrow 0 \]
in $\Coh_{(1)}(X)$ shows $\tilde{\Psi}(D) = \cl \ho_X(D)- \cl\ho_X$ for every effective Weil divisor $D$. If $f=s/t$ is a rational function given by a quotient of two nonzero sections $s,t\in \Ho^0(X,\ho_X(D))$, we get $\tilde{\Psi}(\divisor(f))=\tilde{\Psi}(\divisor(s))-\tilde{\Psi}(\divisor(t))=0$. Thus, we obtain a group homomorphism $\Psi:\Pic(X) \longrightarrow \Ka(\Coh_{(1)}(X))$ mapping a line bundle $L$ to $\cl L-\cl\ho_X$. The morphism $\Psi$ maps $\Pic(X)$ onto a direct summand of $\Ka(\Coh_{(1)}(X))$ because $\det:\Ka(\Coh_{(1)}(X))\longrightarrow \Pic(X)$ is a left inverse of $\Psi$. The image of $\Psi$ is contained in the kernel of the rank homomorphism $\rk:\Ka(\Coh_{(1)}(X))\longrightarrow \mathbb{Z}$. \\
Due to Serre's theorem, every coherent sheaf $G$ on $X$ fits into a short exact sequence
\begin{equation} 0 \longrightarrow \ho_X(mH)^{\oplus \rk(G)} \longrightarrow G\longrightarrow T\longrightarrow 0 ,
\end{equation}
where $H$ is some fixed ample divisor, $m$ some sufficiently small integer and $T$ is a torsion sheaf. If we regard $T$ as an object in $\Coh_{(1)}(X)$, we can assume that $T$ is a successive extension of torsionfree sheaves $T_i$ on irreducible divisors $D_i$. Repeating the argument with the short exact sequence (\arabic{equation}) with $H|_{D_i}$, we see that $T_i$ is a direct sum of line bundles $\ho_{D_i}(m_iH|_{D_i})$ in $\Coh_{(1)}(X)$. The latter are isomorphic to $\ho_{D_i}$ in $\Coh_{(1)}(X)$ and we see that $\cl T$ is a sum of classes  $\cl\ho_{D_i}$ in $\Ka(\Coh_{(1)}(X))$, i.e.\  contained in the image of $\Psi$. If some object $\cl E - \cl F$ in $\Ka(\Coh_{(1)}(X))$ has rank zero, we get
\[ \cl E -\cl F = \big(\cl \ho_X(mH)^{\oplus r} + \cl T_E\big) - \big(\cl \ho_X(mH)^{\oplus r} + \cl T_F\big) = \cl T_E -\cl T_F \in \im\Psi. \]
Thus, the following short sequence is exact and splits
\[ 0 \longrightarrow \Pic(X) \xrightarrow{\;\Psi\;} \Ka(\Coh_{(1)}(X)) \xrightarrow{\;\rk\;} \mathbb{Z} \longrightarrow 0. \]  
\end{proof}
Note that the category $\Coh_{(1)}(X)$ contains the Serre subcategory $\Coh_{(1)}^1(X)$ of torsion sheaves modulo torsion sheaves supported in codimension greater than one. This category is of finite length and the simple objects are (the structure sheaves of) irreducible divisors on $X$. The quotient category is $\Coh_{(0)}$. The associated short exact sequence of $\Ka$-groups is 
\[ \Weil(X) \longrightarrow \Pic(X)\oplus \mathbb{Z} \longrightarrow \mathbb{Z} \longrightarrow 0,\]
where the map on the left hand side is the usual map associating to each Weil divisor its line bundle. As $K(X)\cong \Hom_{(0)}(\ho_X,\ho_X)$ and $k\cong \Hom_{(1)}(\ho_X,\ho_X)$ (cf.\ Proposition \ref{HomExt} (4)), we can extend this sequence to the following exact sequence
\[ 0\longrightarrow \Aut_{(1)}(\ho_X) \longrightarrow \Aut_{(0)}(\ho_X) \longrightarrow \Weil(X) \longrightarrow \Pic(X)\oplus \mathbb{Z} \longrightarrow \mathbb{Z} \longrightarrow 0.\]
The two automorphism groups can be regarded as the first higher $\Ka$-groups of the categories $\Coh_{(1)}(X)$ resp.\ $\Coh_{(0)}(X)$.\\

The next result is a very important tool for the classifications in Section 4 and 5.

\begin{prop} \label{HomExt}
Let $E,E'$ and $T, T'$ be two torsionfree respectively torsion sheaves on $X$ such that $\supp(T)\cap\supp(T')$ contains no divisor. In the case $\dim(X)\ge 2$ we get
\begin{enumerate}
\item $\Hom_{(1)}(T,T')=0$,
\item $\Hom_{(1)}(\ho_D,\ho_D)=K(D)$, where $D\subset X$ is an irreducible effective divisor on $X$ with function field $K(D)$,
\item $\Hom_{(1)}(T,E)=0$, but $\dim\Hom_{(1)}(E,T)=\infty$,
\item $\Hom_{(1)}(E,E')=\Hom(E\dd,E'^{\vee\vee})$, where $E\dd$ denotes the reflexive hull of $E$ and analogue for $E'$, in particular, $\dim\Hom_{(1)}(E,E')<\infty$,
\item $\Ext^1_{(1)}(T,T')=0$, but $\dim\Ext^1_{(1)}(T,T)=\infty$,
\item $\Ext^1_{(1)}(E,T)=0$, but $\dim\Ext^1_{(1)}(T,E)=\infty$,
\item $\dim \Ext^1_{(1)}(E,E')=\infty$.
\end{enumerate}
For $\dim(X)=1$ the results are the same if we replace $\infty$ by some natural number. 
\end{prop}
\setcounter{claim4}{\value{prop}}
\begin{proof} The proposition is well known for $\dim(X)=1$ as $\Coh_{(1)}(X)=\Coh(X)$ in this case. 
To prove it for $\dim(X)\ge 2$, we use Lemma \ref{limitsequence}, in particular
\[ E_2^{p,q} = \varinjlim_{\codim(X\setminus U)>c} \Ho^p(U,\altE xt^q(F|_U,G|_U)), \]
together with Corollary \ref{Ext}. The vanishing of $\Hom_{(1)}(T,E)$ is obvious. If we choose $U$ such that $U\cap\supp(T)\cap\supp(T')=\emptyset$, the equations $\Hom_{(1)}(T,T')=\Ext^1_{(1)}(T,T')=0$  follow immediately. \\
For $\Hom_{(1)}(\ho_D,\ho_D)=K(D)$ we just mention that every nonempty open subset of $D$ is an intersection of $D$ with an open subset $U\subset X$ of $\codim(X\setminus U)>1$. Thus, every rational function on $D$ is contained in $\Hom(\ho_{D\cap U},\ho_{D \cap U})=\Ho^0(U,\altH om(\ho_D|_U,\ho_D|_U))$ for a suitable $U$. \\
For every torsionfree sheaf $E$ there is a short exact sequence in $\Coh(X)$
\[ 0 \longrightarrow E \longrightarrow E\dd \longrightarrow T'' \longrightarrow 0\]
with $\codim \supp(T'')>1$, hence, $E\cong E\dd$ in $\Coh_{(1)}(X)$ which proves the first part of (4). For reflexive sheaves $F,G$ the restriction map $\Hom(F,G)\longrightarrow \Hom(F|_U,G|_U)$ is an isomorphism for all open subsets $U$ with $\codim(X\setminus U)>1$ and the second part of (4) follows. \\
To prove $\dim\Hom_{(1)}(E,T)=\infty$ we can assume $T=\ho_D$ since every torsion sheaf has a filtration with quotients of this kind, as shown in the proof of Proposition \arabic{section}.\arabic{claim1}. As every torsionfree sheaf is locally free outside a closed subset of codimension two, we can further assume that $E$ is locally free on $U$ and, moreover, that $U\cap D$ is affine. Then, $\dim\Ho^0(U,\altH om(E|_U,T|_U))=\dim \Ho^0(U\cap D, E^\vee)=\infty$ since $\dim (U\cap D)\ge 1$. Using $\codim_D(U\cap D)\ge 1$, these sections of $E^\vee|_{D\cap U}$ cannot vanish by restricting them to smaller open subsets $V\subset U$ and the assertion follows.\\
Using the first part of (5), we can restrict the proof of $\dim\Ext^1_{(1)}(T,T)=\infty$ to the case $T=\ho_D$ for some effective divisor $D$ on $X$. We choose $U\subseteq X$ such that $D$ is non-singular on $D\cap U$. After replacing $X$ by $U$, Remark \arabic{section}.\arabic{cor4} will prove the assertion. The same arguments apply to the case $\dim\Ext^1_{(1)}(T,E)=\infty$ if we further assume that $E$ is locally free on $U$. \\
To show the first part of (6) we can assume that $E|_U$ is locally free and, thus, $\altE xt^q(E|_U,T|_U)=0$ for all $q>0$ follows.
For the proof of the last equation (7) we consider the following part of a long exact sequence
\[ \underbrace{\Hom_{(1)}(E/\torssh(E),E')}_{\dim(\ldots)<\infty\mbox{ \scriptsize by }(4)} \longrightarrow \underbrace{\Ext^1_{(1)}(\torssh(E),E')}_{\dim(\ldots)=\infty\mbox{ \scriptsize by }(6)} \longrightarrow \Ext^1_{(1)}(E,E')\]
and the assertion follows immediately.
\end{proof}
\textbf{Remark.} Note that the first part of (6) implies that in $\Coh_{(1)}(X)$ every coherent sheaf $E$ is a direct sum $\torssh(E)\oplus (E/\torssh(E))$ of its torsion subsheaf and the torsionfree quotient. Thus, the computation of $\Hom_{(1)}$-groups and $\Ext^1_{(1)}$-groups can be reduced to the cases discussed in Proposition \ref{HomExt}.\\

\setcounter{claim}{\value{section}}

%\newpage

\section{The space of stability conditions on $\rm D^b_{(1)}(X)$}

In this section we compute the space $\stab(\DER^b_{(1)}(X))$ of locally-finite numerical stability conditions on $\DER^b_{(1)}(X)$ for an irreducible smooth projective variety $X$ of dimension $\dim(X)>1$. The reader who is only interested in the classification of autoequivalences should read the proofs of Lemma 4.1 and Corollary 4.2 and may skip the rest of this section.\\
The case $\dim(X)=1$ was already studied by  E.\ Macr\`{i} \cite{Macri} and S.\ Okada \cite{Okada}. We will see that our result is a natural generalization of the case of curves of genus $g\ge 1$ (see equation (\ref{curves})). Nevertheless, the proof is quite different since we cannot use Serre duality. \\
First of all we need to specify the notion `numerical'. For this we have to choose a free abelian quotient $\Ka(\DER^b_{(1)}(X))\sur \altN(\DER^b_{(1)}(X))$ of finite rank. By Proposition  \arabic{claim}.\arabic{claim1}  we have $\Ka(\DER^b_{(1)}(X))\cong \mathbb{Z} \oplus \Pic(X)$ and it is very natural to take  
\[ \xymatrix  @R=1cm @C=2cm {\rk\oplus \ce_1:\Ka(\DER^b_{(1)}(X)) \ar@{->>}[r] & \mathbb{Z}\oplus \Num^1(X), }\]
where $\Num^1(X)$ is the Picard-group of $X$ modulo numerical equivalence. To motivate this choice one could remark that the central charge of a numerical stability condition should be constant for each member of a flat family of sheaves. As for every numerical trivial line bundle there is some power which has a deformation to the trivial line bundle $\ho_X$ (see Kleiman's expos\'{e} [\cite{SGA6}, XIII, Theorem 4.6] in SGA6), any central charge which is constant under deformations of sheaves is numerical in our sense. Moreover, there is a short exact sequence
\[ 0 \longrightarrow \altN(\Dr{c+1}(X)) \xrightarrow{\qquad} \altN(\Der(X))=\altN(X) \xrightarrow{\rk\oplus \ce_1} \mathbb{Z}\oplus \Num^1(X)\longrightarrow 0,\]
where $\altN(\ldots)=\Ka(\ldots)/\Ka(\ldots)^\perp$, and the orthogonal complement is taken with respect to the Euler pairing $\chi(E,F)=\sum_i (-1)^i \dim \Ext^i(E,F)$. We will start with the classification of all locally-finite slicings. Note that $\Der(\Coh_{(1)}(X))=\DER^b_{(1)}(X)$ by Proposition \ref{equivcat}.

\subsection{Classification of locally-finite slicings}
The  idea of the classification of slicings is to relate stable objects to indecomposable objects using the following simple observation.
\begin{lemma} \label{indecom}
Let $\altP$ be a slicing on a triangulated category $\altD$ and $E,E'$ two objects in $\altD$. Then, $E\oplus E'\in \altP(\phi)$ for some $\phi\in\mathbb{R}$ if and only if $E\in \altP(\phi)$ and $E'\in \altP(\phi)$.
\end{lemma}
\begin{proof}
Since $\altP(\phi)$ is additive per definition, it remains to show that $E\oplus E'\in \altP(\phi)$ implies $E, E'\in \altP(\phi)$. For this let us consider the Harder--Narasimhan filtrations 
\[ \xymatrix @C=0.3cm @R=0.7cm { 0=E_0 \ar[rr] & & E_1 \ar[rr] \ar[dl] & & E_2 \ar[rr] \ar[dl] & & \ldots \ar[rr] & & E_{m-1} \ar[rr] & & E_m=E \ar[dl] \\
 & A_1 \ar@{-->}[ul] & & A_2 \ar@{-->}[ul] & & & & & & A_m \ar@{-->}[ul] } \] 
with $A_i\in \altP(\phi_i)$ and 
\[ \xymatrix @C=0.3cm @R=0.7cm { 0=E'_0 \ar[rr] & & E'_1 \ar[rr] \ar[dl] & & E'_2 \ar[rr] \ar[dl] & & \ldots \ar[rr] & & E'_{n-1} \ar[rr] & & E'_n=E' \ar[dl] \\
 & A'_1 \ar@{-->}[ul] & & A'_2 \ar@{-->}[ul] & & & & & & A'_n \ar@{-->}[ul] } \]
with $A'_j\in \altP(\phi'_j)$ for $E$ respectively $E'$. We sort the set $\{\phi_1,\ldots,\phi_m,\phi'_1,\ldots, \phi'_n\}$ to obtain a  strictly descending sequence $\psi_1 > \ldots >\psi_p$ and we define for $1\le l\le p$ 
\[ F_l\mathrel{\mathop{:}}=E_{\max\{i\mid \phi_i\ge \psi_l\}}\oplus E'_{\max\{j\mid \phi'_j\ge \psi_l\}}.\]
Here we use the convention $E_{\max\emptyset}=E'_{\max\emptyset}=0$. There are natural distinguished triangles
\[ F_{l-1} \longrightarrow F_l \longrightarrow B_l \longrightarrow F_{l-1}[1] \]
with
\[ B_l=\begin{cases}
        A_i\oplus A'_j & \mbox{ if } \psi_l=\phi_i=\phi'_j \mbox{ for suitable } 1\le i\le m, 1\le j\le n, \\
        A_i & \mbox{ if } \psi_l=\phi_i \mbox{ for some } 1\le i\le m \mbox{ but } \psi_l\neq\phi'_j \mbox{ for all }1\le j\le n,\\
        A'_j & \mbox{ if } \psi_l=\phi'_j \mbox{ for some } 1\le j\le n \mbox{ but } \psi_l\neq\phi_i \mbox{ for all }1\le i\le m.\\
       \end{cases} \]
Thus, we obtain a Harder--Narasimhan filtration 
\[ \xymatrix @C=0.3cm @R=0.7cm { 0=F_0 \ar[rr] & & F_1 \ar[rr] \ar[dl] & & F_2 \ar[rr] \ar[dl] & & \ldots \ar[rr] & & F_{p-1} \ar[rr] & & F_p=E\oplus E' \ar[dl] \\
 & B_1 \ar@{-->}[ul] & & B_2 \ar@{-->}[ul] & & & & & & B_p \ar@{-->}[ul] } \] 
for $E\oplus E'$. Since such a filtration is unique, the assertion follows.
\end{proof}
Using our knowledge about $\DER^b_{(1)}(X)$, we obtain the following corollary.
\begin{cor} \label{coro1}
Let $\altP$ be a slicing on $\DER^b_{(1)}(X)$. Every indecomposable semistable object with respect to $\altP$ is up to isomorphism either a shifted indecomposable torsion sheaf or a shifted indecomposable torsionfree sheaf.
\end{cor}
\begin{proof} Let $E$ be an indecomposable object. Since the homological dimension of $\Coh_{(1)}(X)$ is one, the complex $E$ is isomorphic to the direct sum of its shifted cohomology sheaves (see \cite{HuybFourMuk}, Corollary 3.14). Using the previous lemma, we conclude that each shifted cohomology sheaf is semistable of the same phase. Since $E$ is indecomposable, it coincides with one of its shifted cohomology sheaves up to isomorphism and we can assume $E\in \Coh_{(1)}(X)$. Using the remark following Proposition \ref{HomExt}, the previous lemma as well as the indecomposability of $E$, we get $E\cong \torssh(E)$ or $E\cong (E/\torssh(E))$. The indecomposability of $\torssh(E)$ resp.\ $E/\torssh(E)$ follows once again from the indecomposability of $E$. 
\end{proof}
Note that every stable object is indecomposable by Lemma \ref{indecom}. Hence, any stable object is a shifted stable sheaf which is either torsionfree or a torsion sheaf. \\
Let us assume that the slicing is locally-finite. Thus, every nonzero semistable object has a Jordan--H\"older filtration by stable objects of the same phase.  By  Corollary \ref{coro1} every stable object is (isomorphic to) a shifted stable sheaf. Using the Harder--Narasimhan filtration of a complex and the Jordan--H\"older filtrations of its semistable factors, we can produce a filtration whose quotients are shifted stable sheaves. The existence of complexes with nonzero rank shows that there is a stable torsionfree sheaf $F_0$. 
\begin{lemma} 
If $G_1[n_1]\in \altP(\phi_1)$ and $G_2[n_2]\in \altP(\phi_2)$ for stable sheaves $G_1$ and $G_2$ and $\phi_1\ge\phi_2$ then $n_1\ge n_2$. In particular, $G_1[n_1],G_2[n_2]\in \altP(\phi)$ implies $n_1=n_2$. 
\end{lemma}
\begin{proof} Let us denote the phase of $F_0$ by $\phi_0$. By the last corollary, $G_i$ is either a stable torsion sheaf or a stable torsionfree sheaf of phase $\phi_i-n_i$. Assume the contrary $n_1<n_2$, hence 
 \[ \phi_1-n_1 \ge \phi_2-n_2+1.\]
Since $\Ext^1_{(1)}(G_1,F_0)\neq 0$ by Proposition \arabic{claim}.\arabic{claim4}, we get 
\[ \phi_0+1 \ge \phi_1-n_1.\]
If $G_2$ is a torsion sheaf, the nonvanishing of $\Hom_{(1)}(F_0,G_2)$ and $F_0\not\cong G_2$ imply
\[ \phi_2-n_2+1 > \phi_0+1\]
which is a contradiction to the previous two inequalities. If $G_2$ is torsionfree, we use $\Ext^1_{(1)}(G_1,G_2)\neq 0$ to conclude
\[ \phi_1-n_1\le \phi_2-n_2+1.\]
Combining this with the first inequality yields that $G_1$ and $G_2[1]$ are stable objects in $\altP(\phi_1-n_1)$ with a nontrivial morphism between them. Thus, $G_1\cong G_2[1]$ which is a contradiction.
\end{proof}
Since every semistable object has a filtration by stable objects of the same phase, the lemma shows that every semistable object is a shifted sheaf and the number of shifts does not decrease if we increase the phase. Thus, every Harder--Narasimhan filtration is a refinement of the usual cohomology filtration. But the latter is unique and we conclude that the Harder--Narasimhan filtration of a sheaf is in fact a filtration in $\Coh_{(1)}(X)$. In particular, the simple objects of $\Coh_{(1)}(X)$ are stable with respect to the slicing $\altP$. Thus, the structure sheaf $\ho_D$ of an irreducible effective divisor is stable of some phase $\phi_D$ and  for every torsionfree stable sheaf $F$ of phase $\phi$ we conclude 
\[ \phi_D-1 < \phi < \phi_D  \]
from $ \Hom_{(1)}(F,\ho_D)\neq 0$ and $ \Ext^1_{(1)}(\ho_D,F)\neq 0$. If we define $\psi$ by
\[ \psi+1=\sup\{ \phi_D \mid D\subset X \mbox{ an irreducible effective divisor } \},\]
we conclude $\phi\in [\psi,\psi+1)$ and $\phi_D\in(\psi,\psi+1]$ from the upper inequalities and the existence of a stable torsionfree sheaf, e.g.\ $F_0$. If $\psi$ is a phase of a torsionfree stable sheaf $F$, this sheaf has no nontrivial subsheaves, i.e.\ it is simple in $\Coh_{(1)}(X)$. As every torsionfree sheaf has subsheaves, we get $\phi \in (\psi,\psi+1]$ for the phase of every stable sheaf. Using the fact that any sheaf has a filtration by stable sheaves and $\Coh_{(1)}(X)$ as well as $\altP((\psi,\psi+1])$ are hearts of bounded t-structures, we obtain the following result.
\begin{prop} \label{slicings}
For every locally-finite slicing $\altP$ there is a unique $\psi\in\mathbb{R}$ such that $\altP((\psi,\psi+1])=\Coh_{(1)}(X)$.
\end{prop}
\vspace{0.5cm}

\subsection{Classification of stability conditions}

Proposition \ref{slicings} allows the classification of all locally-finite numerical stability conditions on $\DER^b_{(1)}(X)$. Indeed, by applying some element of $\Glt$ to a locally-finite numerical stability condition $\sigma=(Z,\altP)$, we can assume $\altP((0,1])=\Coh_{(1)}(X)$. Using the intersection pairing $\Num^1(X)_\mathbb{R}\times \Num_1(X)_\mathbb{R} \longrightarrow  \mathbb{R}$, we  find two elements $\beta,\omega\in \Num_1(X)_\mathbb{R}$ and two numbers $a,b\in \mathbb{R}$ such that
\[ Z(E)=-\omega.\ce_1(E) +a\rk(E) + i(b\rk(E)+ \beta.\ce_1(E)) \qquad\mbox{ for all }E\in \DER^b_{(1)}(X). \]
If $\beta\neq 0$, there is a line bundle $L$ with $\beta.\ce_1(L)<-b$. In particular, $\Ima Z(L)<0$ which contradicts the axioms of a central charge. Thus, $\beta=0$ and $Z(T)\in \mathbb{R}_{<0}$ for every torsion sheaf $T$ follows. Hence, every torsion sheaf is semistable of the same phase $\phi=1$. Since $\Coh_{(1)}(X)$ is not of finite length, we conclude $b\neq 0$ and after applying a suitable element of $\Glt$ we can assume 
\[ Z(E)=-\omega.\ce_1(E) +i\rk(E) \qquad\mbox{ for all }E\in \DER^b_{(1)}(X). \]
Using $Z(\ho_D)<0$, we get $\omega.D>0$ for every effective divisor $D$ on $X$, where we used the shorthand $D=\ce_1(\ho_D)$. The condition of locally-finiteness forces us to improve the last inequality.
\begin{prop}
Using the previous notation we get 
\[ \inf\{ \omega.D\mid D\subset X \mbox{ an effective divisor on }X \}>0.\]
Conversely, every $\omega\in \Num_1(X)_\mathbb{R}$ with this property defines a locally-finite numerical stability condition on $\DER^b_{(1)}(X)$ with heart $\Coh_{(1)}(X)$ and central charge $Z(E)=-\omega.\ce_1(E) +i\rk(E)$. 
\end{prop}
\textbf{Remarks.} 
\begin{itemize}
\item Obviously, the set $C(X)$ of those $\omega\in \Num_1(X)_\mathbb{R}$ satisfying the inequality $\inf\{ \omega.D\mid D\subset X \mbox{ an effective divisor on }X \}>0$ is a convex cone and, therefore, connected. Furthermore, it is contained in the dual of the pseudoeffective cone $\overline{\Eff(X)}\subset \Num^1(X)_\mathbb{R}$ (cf. \cite{Lazarsfeld1}, Remark 2.2.28). In the case of surfaces the dual of $\overline{\Eff(X)}$ is the nef cone of the surface. On the other hand, by Kleiman's criterion the cone $C(X)$ contains the ample cone. Thus
\[ \Amp(X)= \Int C(X) \subset C(X) \subset \overline{C(X)}=\Nef(X). \]
Using ruled surfaces, one can find examples, where the inclusions are strict.
\item Every rational $\omega$, i.e.\ $\omega\in \Num_1(X)_\mathbb{Q}$, satisfying $\omega.D>0$ for any effective divisor $D$ on $X$ is contained in $C(X)$. Indeed, some positive multiple $n\omega$ of $\omega$ is integral and $\omega.D\ge 1/n$ follows. Note that these $\omega$ are dense in $C(X)$.
\item There is no numerical stability condition $\sigma=(Z,\altP)$ on $\Der(X)$ with heart $\Coh(X)$ for a surface $X$ with a curve $C$. Indeed, as the sheaves $k(x)$ are simple objects in $\Coh(X)$, there are stable and we can assume $k(x)\in \altP(1)$ after applying some element of $\Glt$ on $\sigma$. Hence, $\Ima Z$ cannot depend on $\ce_2$, and we can conclude $\Ima Z =b\cdot\rk $ by the same arguments as before. Thus $\ho_C\in \altP(1)$. Take $x\in C$ and consider the sequence $0\longrightarrow \ho_C(-mx) \longrightarrow \ho_C \longrightarrow \ho_{mx} \longrightarrow 0$, where we regard $mx$ as a divisor on $C$. Since $Z(\ho_{mx})=mZ(k(x))$ and $Z(k(x))<0$, we conclude $Z(\ho_C(-mx))\in \mathbb{R}_{>0}$ for $m\gg 0$ which is a contradiction. Thus, the situation on $\Der(X)$ differs completely from the one on $\DER^b_{(1)}(X)$.  
\end{itemize}
\begin{proof}
Assume the contrary and choose a sequence $(D_j)_{j\in \mathbb{N}}$ of effective divisors such that $\sum_{j\in \mathbb{N}} \omega.D_j$ converges. Since $\sigma$ is locally-finite, there is a real number $\eta>0$ such that $\altP((1-\eta,1+\eta))$ is of finite length. Pick a `sufficiently large' effective divisor $D$ such that $\omega.D-\sum_{j\in \mathbb{N}}\omega.D_j > \cot(\pi\eta)$. It is easy to see that the sheaves $\ho_X(D-\sum_{j=1}^n D_j)$ are stable of phase $\phi_n\in (1-\eta,1)$. Moreover, they form a strictly descending sequence of strict subobjects of $\ho_X(D)$ which contradicts the locally-finiteness of $\altP((1-\eta,1+\eta))$. \\
The  second part of the lemma is more involved. First of all we mention that the assumption on $\omega$ assures the existence of a numerical stability condition $\sigma$ with heart $\Coh_{(1)}(X)$ and central charge $Z(E)=-\omega.\ce_1(E) +i\rk(E)$. Indeed, to show the Harder--Narasimhan condition it is enough to check the properties ii.1 and ii.2 of Section \arabic{claim3}. Since $\Coh_{(1)}(X)$ is noetherian, ii.2 follows. Consider a strictly descending sequence $\ldots \subset E_{j+1}\subset E_j \subset \ldots \subset E_0$. For $j\gg 0$ we have $\rk(E_{j+1})= \rk(E_j)$ and, thus, $Z(E_j/E_{j+1})\in \mathbb{R}_{<0}$. Hence, $\phi_{j+1}<\phi_j$ and ii.1 follows. It remains to show that $\sigma$ is locally-finite which is the hard part of the proof. For this we choose some $\eta\in (0,1/2)$ and define $\varepsilon > 0$ to be
\[ \varepsilon \mathrel{\mathop{:}}= \inf\{ \omega.D\mid D\subset X \mbox{ an effective divisor on }X \}.\]
It suffices to show that $\altP((\phi-\eta,\phi+\eta))$ is locally-finite for all $\phi\in (\eta,\eta+1]$. Objects of $\altP((\phi-\eta,\phi+\eta))$ are complexes $E$ of length two with $H^0(E)\in \altP((\phi-\eta,1])$ and $H^{-1}(E)\in \altP((0,\phi+\eta-1))$, where we use the convention $\altP(\emptyset)=0$. In particular, $H^{-1}(E)$ is torsionfree. For every  short exact sequence $0\longrightarrow E \longrightarrow F \longrightarrow G\longrightarrow 0$ in $\altP((\phi-\eta,\phi+\eta))$ we obtain the long exact cohomology sequence
\[ 0 \longrightarrow H^{-1}(E) \longrightarrow H^{-1}(F) \longrightarrow H^{-1}(G)\longrightarrow H^0(E) \longrightarrow H^0(F) \longrightarrow H^0(G) \longrightarrow 0.\]
To prove that $\altP((\phi-\eta,\phi+\eta))$ is artinian, we choose a descending chain $\ldots \subset E_2\subset E_1 \subset E$ of strict subobjects. Looking at the long exact sequence, we get $\rk(H^{-1}(E_{n+1}))=\rk(H^{-1}(E_n))$ for $n\gg 0$ and since $H^{-1}(E_n/E_{n+1})$ is torsionfree, we obtain isomorphisms $H^{-1}(E_{n+1})\cong H^{-1}(E_n)$. We fix such an integer $n$ with $H^{-1}(E_l)\stackrel{\sim}{\longrightarrow} H^{-1}(E_m)$ for all $l\ge m\ge n$ and introduce the following shorthands for $l>m\ge n$
\begin{eqnarray*} &&L_m \mathrel{\mathop{:}}=H^0(E_m)\in \altP((\phi-\eta,1]), \quad I^l_m\mathrel{\mathop{:}}=\im (L_l \longrightarrow L_m) \in \altP((\phi-\eta,1]) ,\\ 
&&K^l_m\mathrel{\mathop{:}}=H^{-1}(E_m/E_l) \in \altP((0,\phi+\eta-1)),\;\; Q^l_m\mathrel{\mathop{:}}=H^0(E_m/E_l)\in \altP((\phi-\eta,1]). \end{eqnarray*}
The idea is to show $K^l_m=Q^l_m=0$, i.e. $E_l=E_m$ for $l>m\gg 0$. \\
Using the assumption on $n$ and the long exact sequence, we obtain short exact sequences
\[ 0 \longrightarrow K^l_m \longrightarrow L_l \longrightarrow I^l_m \longrightarrow 0\quad \mbox{and}\quad 0 \longrightarrow I^l_m \longrightarrow L_m \longrightarrow Q^l_m \longrightarrow 0.\]
Since $I_m^{l+1} \subset I_m^l$, we get $\rk(I^{l+1}_m) = \rk(I_m^l)$ for $l\gg m$. Thus, $Z(I_m^{l+1})-Z(I_m^{l})\ge \varepsilon$ if $I_m^{l+1}\neq I_m^l$ and $l\gg m$. If the sequence $\ldots \subset I_m^{l+1}\subset I_m^l \subset \ldots \subset I_m^{m+1}$ is not stationary, this contradicts $I_m^l\in \altP((\phi-\eta,1])$ for all $l>m$. Hence $I^l_m=I^{l+1}_m$ for all $l\gg m$ and we denote this subsheaf of $L_m$ by $I_m^\infty$ and the quotient  $L_m/I_m^\infty=\mathrel{\mathop{:}}Q^\infty_m$ coincides with $Q_m^l$ for $l\gg m$. \\
If the set $\{ \rk(K^l_m) \mid l>m\ge n\}$ is not bounded, we can find sequences $(l_p)_{p\in \mathbb{N}}$ and $(m_p)_{p\in \mathbb{N}}$ such that $l_p>m_p\ge n$ and $\lim_{p\rightarrow \infty} \rk(K^{l_p}_{m_p})=\infty$ as well as $\lim_{p\rightarrow \infty}  l_p=\infty$. The snake lemma applied to 
\[ \xymatrix @R=1cm @C=1cm {  0 \ar[r] & K_n^{l_p} \ar[r] \ar[d] & L_{l_p} \ar[d] \ar[r] & I_n^{l_p} \ar[r] \ar@{^{(}->}[d] & 0 \\ 0 \ar[r] & K_n^{m_p} \ar[r] & L_{m_p}  \ar[r] & I_n^{m_p}  \ar[r] & 0  } \]
shows that $K_{m_p}^{l_p}$ is a subsheaf of $K_n^{l_p}$ and $\lim_{p\rightarrow \infty} \rk(K^{l_p}_n)=\infty$ follows. For $p\gg 0$ we have the following short exact sequence
\[ 0 \longrightarrow K_n^{l_p} \longrightarrow L_{l_p} \longrightarrow I_n^\infty \longrightarrow 0 \]
which yields $Z(L_{l_p})=Z(I_n^\infty)+Z(K_n^{l_p})$. Since $K_n^{l_p}\in \altP((0,\phi+\eta-1))$ and $\lim_{p\rightarrow \infty} |Z(I_n^\infty)/Z(K_n^{l_p})|=0$, the phase of $Z(L_{l_p})$ is contained in $(0,\phi+\eta-1+\epsilon)$ for $p\gg0$ depending on $\epsilon>0$. If we choose $\epsilon>0$ such that $\phi+\eta-1+\epsilon<\phi-\eta$, we obtain a contradiction to $L_{l_p}\in \altP((\phi-\eta,1])$. Hence $\rk(K^l_m) \le C$ for all $l>m\ge n$, where $C\in \mathbb{N}$ is some constant. For $l>m\ge n$ we consider the following diagram with exact rows
\[ \xymatrix @C=1cm @R=1cm { 0 \ar[r] & I^\infty_l \ar@{->>}[d] \ar[r] & L_l \ar[d] \ar[r] & Q_l^\infty \ar[d] \ar[r] & 0 \\ 
 0 \ar[r] & I_m^\infty \ar[r] & L_m \ar[r] & Q_m^\infty \ar[r] & 0\,.} \]
By the snake lemma we get the following exact sequence for the kernels of the vertical maps
\[ 0 \longrightarrow K_m^l\cap I_l^\infty \longrightarrow K_m^l \longrightarrow P^l_m \longrightarrow 0\]
with $P^l_m=Q^\infty_l\in \altP((\phi-\eta,1])$ for $l\gg m$ as $\im(L_l \rightarrow L_m)=I^\infty_m$. In particular, using $m=n$ we get $\rk( I^\infty_l)\le \rk(K^l_n) +\rk(I^\infty_n) \le C+\rk(I_n^\infty)=C'$ for all $l>n$. Since the sequence $\rk(I^\infty_l)$ increases with $l$, we conclude $K_m^l\cap I_l^\infty=0$ for all $l>m\gg n$ because $K_m^l$ is torsionfree. Thus, for $l\gg m\gg n$ we get $K_m^l\cong Q_l^\infty$ and $K_m^l=Q^\infty_l=0$ follows since $\altP((0,\phi+\eta-1))\cap \altP((\phi-\eta,1])=0$. This shows $E_m/E_l=0$ for $l\gg m\gg n$ and we conclude $E_{q+1}=E_q$ for $q\gg 0$. Thus, $\altP((\phi-\eta,\phi+\eta))$ is artinian.\\
\newline
To show that $\altP((\phi-\eta,\phi+\eta))$ is noetherian, we consider an ascending chain $E_1 \subset E_2 \subset \ldots\subset E$ of strict subobjects of $E$ and try to find arguments similar to the artinian case. By the long exact cohomology sequences we obtain an ascending sequence $H^{-1}(E_1) \subset H^{-1}(E_2) \subset \ldots \subset H^{-1}(E)$ and conclude $H^{-1}(E_n)=H^{-1}(E_{n+1})=\mathrel{\mathop{:}} H$ for $n\gg 0$ since $\Coh_{(1)}(X)$ is noetherian. In particular, this yields the following exact sequences for $n>m\gg 0$
\[ 0 \longrightarrow H^{-1}(E_n/E_m) \longrightarrow H^0(E_m) \longrightarrow H^0(E_n) \longrightarrow H^0(E_n/E_m) \longrightarrow 0.\]
Let us introduce the following shorthands for $n>m\gg 0$
\begin{eqnarray*} &&L_m \mathrel{\mathop{:}}=H^0(E_m)\in \altP((\phi-\eta,1]), \quad I^m_n\mathrel{\mathop{:}}=\im (L_m \longrightarrow L_n) \in \altP((\phi-\eta,1]) ,\\ 
&&K^m_n\mathrel{\mathop{:}}=H^{-1}(E_n/E_m) \in \altP((0,\phi+\eta-1)),\; Q^m_n\mathrel{\mathop{:}}=H^0(E_n/E_m)\in \altP((\phi-\eta,1]). \end{eqnarray*}
Thus, the upper four-term sequence splits into two short exact sequences
\[ 0 \longrightarrow K^m_n \longrightarrow L_m \longrightarrow I^m_n \longrightarrow 0\quad \mbox{and}\quad 0 \longrightarrow I^m_n \longrightarrow L_n \longrightarrow Q^m_n \longrightarrow 0.\]
Since $\Coh_{(1)}(X)$ is noetherian, we get $K^m_n=K^m_{n+1}=\mathrel{\mathop{:}}K^m_\infty$ and $I^m_n=I^m_{n+1}=\mathrel{\mathop{:}}I^m_\infty$ for $n\gg m$. The natural maps $I^m_p \longrightarrow I^n_q$ for $q\ge p >n\ge m$ induce a morphism $I^m_\infty \longrightarrow I^n_\infty$ and we obtain the following diagram with exact rows for $q\ge p \gg n\ge m$
\[ \xymatrix @R=1cm @C=1cm { 0 \ar[r] & I^m_\infty \ar[d] \ar[r] & L_p \ar[d] \ar[r] & Q^m_p \ar[d] \ar[r] & 0 \\ 0 \ar[r] & I^n_\infty \ar[r] & L_q \ar[r] & Q^n_q \ar[r] & 0\,. } \]
In the case $q > p\gg m=n$  the snake lemma implies the following exact sequence
\begin{equation} \label{eqQ} 0 \longrightarrow K^p_q \longrightarrow Q^m_p \longrightarrow Q^m_q \longrightarrow Q^p_q \longrightarrow 0\,.
\end{equation}
On the other hand, using the special case $q=p\gg n\ge m$ we see that $I^m_\infty \longrightarrow I^n_\infty$ is a monomorphism and we get the short exact sequence
\[ 0 \longrightarrow I^n_\infty/I^m_\infty \longrightarrow Q^m_q \longrightarrow Q^n_q \longrightarrow 0\,.\]
If we replace $p$ by $n$ in the first case and combine the result with the second special case, we conclude 
\begin{equation} \label{eqI} I^n_\infty/I^m_\infty \cong Q^m_n/K^n_q  \qquad\mbox{for all } q>n\gg m .
\end{equation}
To proceed we consider the long exact cohomology sequence associated to $0 \longrightarrow E_n \longrightarrow E \longrightarrow E/E_n \longrightarrow 0$ and use $H^{-1}(E_n)=H$ for all $n\gg 0$
\[ 0 \longrightarrow H \longrightarrow H^{-1}(E) \longrightarrow H^{-1}(E/E_n)\longrightarrow L_n \longrightarrow H^0(E) \longrightarrow H^0(E/E_n) \longrightarrow 0.\]
Since $\Coh_{(1)}(X)$ is noetherian, we get $Z(H^{-1}(E/E_n))=Z(L_n)+ Z(H^{-1}(E))-Z(H)-Z(I)$ for all $n\gg 0$, where $I\subset H^0(E)$ is the `limit' of the ascending chain of the images of $L_n$ in $H^0(E)$. This equation shows $|Z(L_n)|\le C$ for all $n$, where $C$ is a constant. Indeed, if $|Z(L_n)|$ is not bounded, we obtain $Z(H^{-1}(E/E_{n_p}))/Z(L_{n_p}) \xrightarrow{p\rightarrow \infty} 1$  for a subsequence $(n_p)_{p\in \mathbb{N}}$ which leads to  a contradiction to $H^{-1}(E/E_{n_p})\in \altP((0,\phi+\eta-1))$ and $L_{n_p}\in \altP((\phi-\eta,1])$ as in the artinian case. Using the exact sequence
\[ 0 \longrightarrow I^m_\infty \longrightarrow L_n \longrightarrow Q^m_n \longrightarrow 0\] 
for $n\gg m$, we conclude $|Z(I^m_\infty)|\le C'$ for all $m$, because otherwise there is a subsequence $I^{m_p}_\infty$ with $\Rea Z(I^{m_p}_\infty) \longrightarrow -\infty$ and, thus, $\Rea Z(Q^{m_p}_{n_p})\longrightarrow +\infty$ which contradicts $Q^{m_p}_{n_p}\in \altP((\phi-\eta,1])$ using $\rk(Q^{m_p}_{n_p})\le C$. As $I^m_\infty \subset I^{m+1}_\infty$, we get $\rk(I^{m}_\infty)=\rk(I^{m+1}_\infty)$ for $m\gg 0$. If the sequence $(I^m_\infty)_{m\in \mathbb{N}}$ is not stationary, we obtain the contradiction $\Rea Z(I^n_\infty)\longrightarrow -\infty$ to $|Z(I^m_\infty)|\le C'$ because $Z(I^{m+1}_\infty/I^m_\infty)\le -\varepsilon$ for all $m\gg 0$ with $I^{m}_\infty\neq I^{m+1}_\infty$. Thus, by equation (\ref{eqI}) $Q^m_n=K^n_q$ for all $q>n\gg m$ and $m\gg 0$ fixed. Using $K^n_q\in \altP((0,\phi+\eta-1))$ and $Q^m_n\in \altP((\phi-\eta,1])$, we get $Q^m_n=0$ for all $n\gg m$ and by (\ref{eqQ}) $K^p_q=Q^p_q=0$, i.e.\ $E_p = E_q$ for all $q>p\gg 0$. Thus, $\altP((\phi-\eta,\phi+\eta))$ is noetherian.
\end{proof}
Finally, we get the following theorem by combining the previous two propositions.
\begin{theorem} \label{classification}
In the $\Glt$-orbit of every locally-finite numerical stability condition on $\DER^b_{(1)}(X)$ there is a stability condition with heart $\Coh_{(1)}(X)$ and central charge $Z(E)=-\omega.\ce_1(E)+ i\rk(E)$, where $\omega\in \Num_1(X)_\mathbb{R}$ is determined by the orbit up to some positive scalar $r\in \mathbb{R}$. The set of all $\Glt$-orbits in $\stab(\DER^b_{(1)}(X))$ is parametrized by the rays in the convex cone 
\[ C(X)=\big\{\omega\in \Num_1(X)_\mathbb{R} \mid \inf\{ \omega.D\mid D\subset X \mbox{ an effective divisor } \}>0 \big\}.\]
\end{theorem}
\begin{proof} It remains to show that $Z(E)=-\omega.\ce_1(E) +i\rk(E)$ and $Z'(E)=-\omega'.\ce_1(E) +i\rk(E)$ are in the same $\Gl2$-orbit if and only if $\omega'=r\,\omega$ for some $r\in \mathbb{R}_{>0}$. This easy calculation is left to the reader. 
\end{proof}
\textbf{Remark.} If the reader looks carefully at the previous proofs, they will realize that we have not used the fact that the stability condition is numerical. Thus, we have classified all locally-finite stability condition on $\DER^b_{(1)}(X)$. Theorem \ref{classification} generalizes literally if we replace $C(X)$ by the cone
\[ \big\{ \omega\in \Hom(\Pic(X),\mathbb{R}) \mid \inf\{\omega(\ho_X(D))\mid D\subset X \mbox{ an effective divisor } \}>0 \big\}.\]
\vspace{0.2cm}

\subsection{The topology of the space of stability conditions}
 
This subsection is devoted to the topology of the space $\stab(\DER^b_{(1)}(X))$. It turns out that every orbit $\sigma\cdot\Glt$ is a connected component. \\
Let us consider a connected component $\Sigma$ of $\stab(\DER^b_{(1)}(X))$. By Theorem \ref{bigtheorem} there is a complex linear space $V(\Sigma)\subseteq \Hom_\mathbb{Z}(\mathbb{Z}\oplus \Num^1(X),\mathbb{C})$ such that $\pi:\Sigma \rightarrow V(\Sigma)$ is a local homeomorphism. If we fix a stability condition $\sigma=(Z,\altP)\in \Sigma$, the space $V(\Sigma)$ is given by 
\[ V(\Sigma)=\{U\in \Hom_\mathbb{Z}(\mathbb{Z}\oplus \Num^1(X),\mathbb{C})\mid  \|U\|_\sigma < \infty\}, \] 
where
\[ \| U \|_\sigma \mathrel{\mathop{:}}= \sup\left\{ \left. \frac{|U(E)|}{|Z(E)|} \;\right| E \mbox{ semistable in } \sigma\right\}. \]
Assume that $V(\Sigma)$ and, hence, $\Sigma$ are of complex dimension two. As the $\Glt$-orbit of $\sigma$ has four real dimensions, the orbit is open in $\Sigma$. Furthermore, the image of the central charge $W$ of a stability condition $\tau=(W,\altQ)$ in the boundary of $\sigma\cdot\Glt$  is contained in a real line in $\mathbb{C}$ since otherwise $\pi(\tau\cdot \Glt)=W\cdot \Gl2$ is open in $V(\Sigma)$ and, therefore, not contained in the boundary of $Z\cdot \Gl2$. As this contradicts Theorem \ref{classification}, $\sigma\cdot\Glt$ is also closed in $\Sigma$ and 
\[ \sigma\cdot \Glt = \Sigma\]
follows. It remains to show $\dim_\mathbb{C} V(\Sigma) =2$ for every connected component $\Sigma$ of $\stab(\DER^b_{(1)}(X))$. Since $\dim_\mathbb{R} \sigma\cdot \Glt=4$ for every $\sigma=(Z,\altP)\in \Sigma$ by Theorem \ref{classification}, it suffices to find a contradiction if $\dim_\mathbb{C} V(\Sigma) >2$. Let us start with the case $\altP((0,1])=\Coh_{(1)}(X)$ and $Z(E)=-\omega.\ce_1(E)+i\rk(E)$ with rational $\omega$ before considering the general case. If $\dim_\mathbb{C} V(\Sigma) >2$, there is another stability condition $\sigma'=(Z',\altP')\in \Sigma$ not contained in the $\Glt$-orbit of $\sigma$. We can assume $Z'(E)=-\omega'.\ce_1(E) + i\rk(E)$ and $\omega'\notin \mathbb{R}\omega$ by Theorem \ref{classification}. Since the intersection product $\Num_1(X)\times \Num^1(X) \longrightarrow \mathbb{Z}$ is non-degenerate and $\omega\in \Num_1(X)_\mathbb{Q}=\Num_1(X)\otimes \mathbb{Q}$, there is an integral divisor $D$ on $X$ such that $\omega.\ce_1(\ho_X(D))=0$ but $\omega'.\ce_1(\ho_X(D))=\mathrel{\mathop{:}}\Delta>0$. Obviously, $\ho_X(mD)$ is $\sigma$-semistable for all $m\in \mathbb{N}$ and, therefore,
\[ \frac{|Z'(\ho_X(mD))|}{|Z(\ho_X(mD))|}=\sqrt{\frac{(m\Delta)^2+1}{1}} \le \| Z'\|_\sigma<\infty \qquad \mbox{for all }m\in \mathbb{N} \]
which is a contradiction. In the general case $\omega\in C(X)$ the divisor $D$ is just an $\mathbb{R}$-divisor, but we can use the fact that $\Num^1(X)_\mathbb{Q}$ is dense in $\Num^1(X)_\mathbb{R}$ to construct $\mathbb{Q}$-divisors $D_m$ in the neighbourhood of $mD$ such that $|\omega.D_m|<\delta$ and $\omega'.D_m > m\Delta-\delta>0$ for all $m\in \mathbb{N}$, where $0<\delta<\Delta$ is some small real number. If we choose $r_m\in \mathbb{N}$ such that $r_mD_m$ is integral, we obtain
\[ \| Z'\|_\sigma \ge \frac{|Z'(\ho_X(r_mD_m))|}{|Z(\ho_X(r_mD_m))|}\ge \sqrt{\frac{(r_mm\Delta-r_m\delta)^2+1}{(r_m\delta)^2+1}}  \qquad \mbox{for all }m\in \mathbb{N} \]
contradicting $\| Z'\|_\sigma<\infty$. Thus, we have proved the following theorem.
\begin{theorem}
Let $X$ be an irreducible smooth projective variety of dimension $\dim(X)\ge 2$ and $\DER^b_{(1)}(X)$ the quotient category of $\Der(X)$ by the full subcategory of complexes supported in codimension $c>1$. Then, $\Glt$ acts free on $\stab(\DER^b_{(1)}(X))$ and any $\Glt$-orbit is a connected component of the complex manifold $\stab(\DER^b_{(1)}(X))$. The space of connected components is parametrized by the set of rays in the convex cone 
\[ C(X)=\big\{\omega\in \Num_1(X)_\mathbb{R} \mid \inf\{ \omega.D\mid D\subset X\mbox{ an effective divisor on }X \}>0 \big\}.\]
For each $\omega\in C(X)$ there is a unique stability condition in the component associated to $\mathbb{R}_{>0}\omega$ with heart $\Coh_{(1)}(X)$ and central charge $Z(E)=-\omega.\ce_1(E)+i\rk(E)$. 
\end{theorem}
Since  $\Glt$-orbits are always connected, $\stab(\DER^b_{(1)}(X))$ is as disconnected as it could be. In contrast to this, the parameter space of $\Glt$-orbits which is the set of rays in $C(X)$ is connected.  \\

\section{Birational geometry and $\rm D^b_{(1)}(X)$}

In the last section we classify all exact $k$-equivalences between quotient categories $\DER^b_{(1)}(X)$ and $\DER^b_{(1)}(Y)$ for $\dim(X)\ge 2$. Note that $\DER^b_{(1)}(X)\cong \DER^b_{(1)}(Y)$ and $\dim(X)\ge 2$ implies $\dim(Y)\ge 2$ by Proposition \ref{HomExt}. All varieties are assumed to be irreducible smooth and projective. It turns out that the quotient category $\DER^b_{(1)}(X)$ determines $X$ as an object in $\Var_{(1)}$. At the end of this section we give a short discussion of the case $c>1$ and prove the non-existence of a Serre functor on the quotient category if $\dim(X)\ge 2$.\\ 
We start our classification of exact $k$-equivalences $\Psi:\DER^b_{(1)}(X) \longrightarrow \DER^b_{(1)}(Y)$ in analogy to the classification of slicings. Note that the standard t-structure is a slicing on $\DER^b_{(1)}(X)$ which is mapped to another slicing $\altP$ on $\DER^b_{(1)}(X)$ by $\Psi$. To be precise, $\altP(\phi)=\Psi(\Coh_{(1)}(X))[\phi]$ for $\phi\in \mathbb{Z}$ and $\altP(\phi)=0$ otherwise. In particular, Corollary \ref{coro1} is valid and can be written as follows in our situation.
\begin{cor} 
Let $\altA=\altP(0)$ be the image of $\Coh_{(1)}(X)$. Every indecomposable object in the abelian category $\altA$ is up to isomorphism either a shifted indecomposable torsion sheaf or a shifted indecomposable torsionfree sheaf on $Y$.
\end{cor}
Unfortunately, both slicings are not locally-finite and we have to proceed in a different way. Note that $\altA$ has homological dimension one in $\DER^b_{(1)}(Y)$ as it is isomorphic to the heart $\Coh_{(1)}(X)$ of the standard t-structure on $\DER^b_{(1)}(X)$. In particular, every complex is a direct sum of shifted objects in $\altA$ and objects in $\altA$ are direct sums of indecomposable objects since $\Coh_{(1)}(X)$ has this property. Thus, every indecomposable sheaf is up to a shift contained in $\altA$ which is the converse statement of the upper corollary. We will show that the number of shifts is independent of the indecomposable sheaf. Indeed, let $F$ be an indecomposable torsion free sheaf and if we replace $\Psi$ by $[m]\circ\Psi$, we can assume $F\in \altA$. Let $G$ be another indecomposable sheaf which is, therefore, contained in $\altA[n]$ for some integer $n$. Using $\Ext^1_{(1)}(G,F)=\Hom_{(1)}(G,F[1])\neq 0$ (cf.\ Prop.\ \ref{HomExt}), we get $n\le 1$, and since $\altA$ has homological dimension one, we conclude $n\in\{0,1\}$. If $G$ is torsionfree, $\Ext^1_{(1)}(F,G)\neq 0$ yields $n=0$ by the same argument. If $G$ is a torsion sheaf, we argue as follows to exclude the case $n=1$. Assume the contrary and write $F=\Psi(E_1)$ and $G=\Psi(E_2)[1]$ for some indecomposable sheaves $E_1,E_2\in\Coh_{(1)}(X)$. As
\begin{eqnarray*} &\infty>\dim\Hom_{(1)}(F,F)=\dim\Hom_{(1)}(E_1,E_1),& \\ 
&\infty=\dim\Hom_{(1)}(G,G)=\dim\Hom_{(1)}(E_2,E_2),&\end{eqnarray*}
we see that $E_1$ is torsionfree and $E_2$ a torsion sheaf (cf.\ Prop. \ref{HomExt}). Hence,
\[ \infty=\dim \Ext^1_{(1)}(G,F)=\dim \Hom_{(1)}(E_2,E_1)=0\]
which is a contradiction. Thus $G\in \altA$ and we have proved the following proposition.
\begin{prop} \label{autoequiv}
For  any exact $k$-equivalence $\Psi:\DER^b_{(1)}(X)\longrightarrow \DER^b_{(1)}(Y)$ there is a unique integer $n$ such that $\Psi(\Coh_{(1)}(X))=\Coh_{(1)}(Y)[n]$.
\end{prop}
Hence, up to shifts, an exact equivalence $\Psi$ is induced by an exact functor $F:\Coh_{(1)}(X)\longrightarrow \Coh_{(1)}(Y)$, i.e.\ $\Psi=[n]\circ \Der(F)$. As structure sheaves of integral divisors are simple objects in $\Coh_{(1)}(X)$, $F$ maps structure sheaves of integral divisors to those sheaves. Since every torsion sheaf is an extension of such sheaves, $F$ maps torsion sheaves to torsion sheaves and induces, therefore, an invertible exact functor $F_0:\Coh_{(0)}(X)\longrightarrow \Coh_{(0)}(Y)$ between the quotient categories of $\Coh_{(1)}(X)$ resp.\ $\Coh_{(1)}(Y)$ by the full Serre subcategories of torsion sheaves. This quotient categories  are equivalent to the categories of finite-dimensional vector spaces over the function fields of $X$ resp.\ $Y$ (cf.\ Prop.\ \ref{case0}). The induced functor $F_0$ is, therefore,  determined by the isomorphism $K(X)\cong\Hom_{(0)}(\ho_X,\ho_X) \longrightarrow \Hom_{(0)}(\ho_Y,\ho_Y)\cong K(Y)$, i.e. by a birational map $\psi:Y\dashrightarrow X$. Note that $\psi$ depends functorially on $\Psi$. In particular, we obtain a group homomorphism $\Aut(\DER^b_{(1)}(X))\longrightarrow \Bir(X)$. It is obvious, that the kernel contains $\Pic(X)\oplus \mathbb{Z}$ acting by tensor products and shifts. It turns out, that $\Psi$ is uniquely determined by $\psi$ up to tensor products with line bundles and shifts.
\begin{theorem} \label{functordecomp}
Let $X$ and $Y$ be two irreducible smooth projective varieties of dimension at least two. Any exact $k$-equivalence $\Psi:\DER^b_{(1)}(X)\longrightarrow \DER^b_{(1)}(Y)$ has a unique decomposition $\Psi=[n]\circ L\circ \psi^\ast$ by a shift functor, a tensor product with a line bundle $L\in \Pic(X)$ and a pullback functor induced by a birational map $\psi:Y \dashrightarrow X$ which is an isomorphism in codimension one.
\end{theorem}
Note that the theorem is still valid for curves of genus $g\neq 1$. The Fourier--Mukai transform with respect to the Poincar\'{e} bundle is an exact equivalence $\Der(X)\cong \Der(\hat{X})$ between the derived categories of elliptic curves without such a decomposition.
Before we prove the theorem, we will state some immediate consequences.
\begin{cor}
Two irreducible smooth projective varieties $X$ and $Y$ are isomorphic in codimension one if and only if their quotient categories $\DER^b_{(1)}(X)$ and $\DER^b_{(1)}(Y)$ are equivalent as $k$-linear triangulated categories.
\end{cor}
For the case $\dim(X)=\dim(Y)\le 1$ see the textbook \cite{HuybFourMuk}. Another way, to formulate the corollary is to say that $X$ and $Y$ are isomorphic in $\Var_{(1)}$ if and only if there is an exact $k$-equivalence between their quotient categories. Thus, the functor $\DER^b_{(1)}(-)$ on $\Var_{(1)}$ contains enough information to classify objects in $\Var_{(1)}$.
\begin{cor} \label{isomorphic}
Two irreducible smooth projective varieties $X$ and $Y$ of dimension $\dim(X)=\dim(Y)\le 2$ are isomorphic if and only if there is an exact $k$-equivalence between their quotient categories $\DER^b_{(1)}(X)$ and $\DER^b_{(1)}(Y)$. 
\end{cor}
This result is well known for curves (cf.\ \cite{HuybFourMuk}). The surprising result is that the quotient category classifies irreducible smooth projective surfaces while the usual derived category does not. There are non-isomorphic K3-surfaces with equivalent derived categories. Roughly speaking, the usual derived category contains to much redundancies which allows `strange' functors.
\begin{cor}
For any irreducible smooth projective variety $X$ of dimension $\dim(X)\ge 2$ there is a natural exact sequence of groups
\[ 0 \longrightarrow \mathbb{Z}[1] \oplus \Pic(X) \longrightarrow \Aut(\DER^b_{(1)}(X)) \longrightarrow \Aut_{(1)}(X)\longrightarrow 0, \]
where the group on the right hand side is the group of all birational automorphisms of $X$ which are isomorphisms in codimension one, i.e. automorphisms of $X$ in the category $\Var_{(1)}$. 
\end{cor}
This statement holds also for curves of genus $g\neq 1$ (cf. \cite{HuybFourMuk}). The correct sequence for  elliptic curves is (cf. \cite{HuybFourMuk}, Section 9.5) 
\[ 0 \longrightarrow 2\mathbb{Z} \times \big(\Aut(X) \ltimes \Pic(X)\big) \longrightarrow \Aut(\Der(C)) \xrightarrow{\ch} \SL(2,\mathbb{Z}) \longrightarrow 0.\] 
\begin{proof}[Proof of Theorem \ref{functordecomp}] We use the notion introduced before stating Theorem \ref{functordecomp}. \\
\textit{Step 1:} Note that $F(\ho_X)$ is torsionfree as it contains no simple objects of $\Coh_{(1)}(Y)$ and, moreover, it has rank one since $F(\ho_X)\cong \ho_Y$ in $\Coh_{(0)}(Y)$. Thus, the sheaf $F(\ho_X)$ is a line bundle outside a closed subset of codimension greater than one and this line bundle extends in a unique way to a line bundle $L$ on $X$. Hence, $\Psi(\ho_X)\cong L[n]$ for uniquely determined $n\in \mathbb{Z}$ and  $L\in \Pic(X)$. After replacing $\Psi$ by $L^{-1}\circ [-n]\circ \Psi$ we can assume $\Psi(\ho_X)= \ho_Y$.\\
\textit{Step 2:} We proceed by showing that $\psi:Y\dashrightarrow X$ is an isomorphism in codimension one. First of all, one can always assume that $\psi$ is defined on an open subset of $Y$ whose complement has codimension at least two. By applying the same arguments to $\psi^{-1}:X \dashrightarrow Y$ induced by $\Psi^{-1}$, it is enough to show that the exceptional locus of $\psi$, i.e.\ the locus, where $\psi$ is either not defined or no isomorphism, has no divisorial part $E$. To prove this, assume there is an exceptional divisor $0\neq E\subset Y$ of $\psi$. By the arguments above, $\psi$ is defined on a dense open subset $E'$ of $E$ and $\codim(\psi(E'))\ge 2$ by the general theory of birational maps. Take two irreducible effective divisors $D_0$ and $D_\infty$ on $X$ contained in the same linear system $|\altL|$ of a line bundle $\altL$ such that $\psi(E')\subset D_0$ and $\psi(E')\cap D_\infty=\emptyset$. Pick two corresponding sections $s_0$ and $s_\infty$ of $\altL$.  Thus, we get a rational function $f=s_0/s_\infty$ on $X$ with $\divisor(f)=D_0-D_\infty$. We let it to the reader to show that the rational function $f\circ \psi=F_0(f)$ is the quotient of the sections $F(s),F(s_\infty)$ of $F(\altL)$ which is again a line bundle (up to isomorphism) as seen in Step 1. The short exact sequence
\[ 0 \longrightarrow \ho_X \xrightarrow{s_0} \altL \longrightarrow \ho_{D_0} \longrightarrow 0\]
in $\Coh_{(1)}(X)$ is mapped by $F$ to the short exact sequence 
\[ 0 \longrightarrow \ho_Y \xrightarrow{F(s_0)} F(\altL) \longrightarrow F(\ho_{D_0}) \longrightarrow 0.\]
Thus, the structure sheaf of the zero divisor of the section $F(s_0)$ is isomorphic to the sheaf $F(\ho_{D_0})$ which is a simple object of $\Coh_{(1)}(Y)$. In other words, the zero divisor of $f\circ \psi$ is irreducible. By contruction it contains the exceptional divisor of $\psi$ and the strict transform of $\altD_0$ which contradicts the irreducibility. Hence, $E=0$ and $\psi$ is an isomorphism in codimension one. \\
The composition $\Psi'\mathrel{\mathop{:}}=(\psi^{-1})^\ast\circ \Psi$ is an exact autoequivalence on $\DER^b_{(1)}(X)$ with $\Psi'(\ho_X)=\ho_X$ and the induced rational map $\psi'$ is the identity. \\
\textit{Step 3:} Let us denote for simplicity the exact $k$-linear autoequivalence $\Psi'$ on $\DER^b_{1}(X)$ of Step 2 by $\Psi$. To prove the theorem we have to show  $\Psi\cong \Id$. Let us consider the sheaf $\Psi(\ho_X(H))$, where $H$ is some effective very ample divisor on $X$ defined by some section $\sigma\in \Ho^0(X,\ho_X(H))$. By the above arguments, $\Psi(\ho_X(H))$ is a line bundle $\ho_X(D)$ for some effective divisor $D$ on $X$ as we have the exact sequence 
\[0 \longrightarrow  \ho_X \xrightarrow{F(\sigma)}\Psi(\ho_X(H))=\ho_X(D) \longrightarrow \Psi(\ho_H) \longrightarrow 0,\]
and the sheaf on the right is a torsion sheaf which we can assume to be $\ho_D$. Take another $H_1\in|H|$ having no common components with $H$ and consider the rational function $f$ with $\divisor(f)=H-H_1$. Denote by $\ho_{D_1}$ the image $\Psi(\ho_{H_1})$. Since $\Psi$ acts as the identity on $K(X)$, we obtain $H-H_1=\divisor(f)=\divisor(\Psi(f))=D-D_1$. As $\Psi$ maps effective divisors to effective divisors, we conclude $\ho_H=\ho_D=\Psi(\ho_H)$ as $D$ has no common components with $D_1$. Thus, $\ho_X(mH)\cong \Psi(\ho_X(mH))$. Without violating our previous assumptions on $\Psi$, we can modify $\Psi$ by a suitable choice of isomorphism $\ho_X(mH)\cong \Psi(\ho_X(mH))$ to assure $\ho_X(mH)= \Psi(\ho_X(mH))$ for all $m\in\mathbb{Z}$ and  $\Psi(\sigma)=\sigma\in \Ho^0(X,\ho_X(H))$. Thus, we obtain an isomorphism of vector spaces
\[ \bigoplus_{m\in \mathbb{Z}} \Ho^0(X,\ho_X(mH)) \xrightarrow{\Psi} \bigoplus_{m\in \mathbb{Z}} \Ho^0(X,\ho_X(mH)).\]
Moreover, this isomorphism is compatible with the product structure. This is a consequence of the fact that the product of two sections $s\in \Ho^0(X,\ho_X(mH))$ and $t\in \Ho^0(X,\ho_X(nH))$ is the diagonal morphism in the following cocartesian square
\[\xymatrix @R=1cm @C=1cm { \ho_X \ar[r]^s \ar[d]_t \ar[dr]^{st} & \ho_X(mH) \ar[d] \\ \ho_X(nH) \ar[r] & \ho_X((n+m)H). }\]
As $\Psi$ maps cocartesian squares to cocartesian squares, this yields a ring isomorphism on $\bigoplus_{m\in \mathbb{Z}} \Ho^0(X,\ho_X(mH))$. Moreover, every section $t\in\Ho^0(X,\ho_X(mH))$ is fixed by $\Psi$ since $\sigma$ as well as the rational function $t/\sigma^m$ are fixed. Hence, $\Psi$ is (isomorphic to) the identity functor on the full subcategory $\altH$ spanned by the sequence $\ho_X(mH), m\in \mathbb{Z}$. \\
\textit{Step 4:} The next step is to extend this isomorphism $\Psi|_\altH\cong \Id_\altH$ to an isomorphism $\Psi\cong \Id$ on $\DER^b_{(1)}(X)$. This is possible, if the sequence $\ho_X(mH)$ is an ample sequence in $\Coh_{(1)}(X)$ by a result of Bondal and Orlov \cite{BondalOrlov01},\cite{Orlov97}. \\
A  sequence of objects $L_m$ in an abelian category $\altA$ is called ample if for every $A\in \altA$ there is an integer $m_0(A)$ such that for all $m\le m_0$ one has:
\begin{itemize}
 \item[(i)] There is an epimorphism $L_m^{\oplus r_m} \sur A$ for some $r_m\in \mathbb{N}$.
 \item[(ii)] $\Hom_\altA(L_m,A[j])=0$ for all $j\neq 0$, and 
 \item[(iii)] $\Hom(A,L_m)=0$.
\end{itemize}
Unfortunately, our sequence $L_m=\ho_X(mH)$ is not ample in $\Coh_{(1)}(X)$ as for any torsion free sheaf $F$ we have $\Ext^1_{(1)}(\ho_X(mH),F)\neq 0$, but it satisfies properties (i) and (iii). Moreover, it is an ample sequence in $\Coh(X)$. Using this, we try to copy the proof of Bondal and Orlov as far as possible following its presentation in Huybrechts textbook \cite{HuybFourMuk} (see Proposition 4.23). The 3rd step therein allows us to extend the isomorphism $\Psi|_\altH\cong \Id_\altH$ to the abelian category $\Coh_{(1)}(X)$. The only argument which does not apply literally is the proof that for a given morphism $A_1\rightarrow A_2$ one can find epimorphisms $L_m^{\oplus k} \sur A_1$ and $L_n^{\oplus l} \sur A_2$ and a morphism $L_m^{\oplus k} \rightarrow L_n^{\oplus l}$ such that the following diagram is commutative
\[ \xymatrix @C=1cm @R=1cm { L_m^{\oplus k} \ar@{->>}[r] \ar[d] & A_1 \ar[d] \\ L_n^{\oplus l} \ar@{->>}[r] & A_2.} \]
To show this we represent the morphism $A_1\rightarrow A_2$ by a roof. Doing this, we find a sheaf homomorphism $A'_1\rightarrow A_1$ in $\Coh(X)$ inducing an isomorphism in $\Coh_{(1)}(X)$ such that the composition $A'_1\rightarrow A_1\rightarrow A_2$ is a sheaf homomorphism in $\Coh(X)$. Take a surjection $L^{\oplus l}_n \sur A_2$ in $\Coh(X)$ and denote the kernel by $B$. As $\Ext^1(L_m^{\oplus k},B)=0$ in $\Der(X)$ for $m\ll 0$, the composition $L^{\oplus k}_m \sur A'_1\rightarrow A_2$ has a lift $L^{\oplus k}_m \rightarrow L^{\oplus l}_n$ for $m\ll 0$ and any surjective sheaf homomorphism $L^{\oplus k}_m \sur A'_1$ in $\Coh(X)$. \\
\textit{Step 5:} To extend the isomorphism $\Psi|_{\Coh_{(1)}(X)}\cong \Id_{\Coh_{(1)}(X)}$ to an isomorphism $\Psi \cong \Id$ on the bounded derived category $\DER^b_{(1)}(X)$, we mention that every complex in $\DER^b_{(1)}(X)$ is isomorphic to the direct sum of its cohomology sheaves. Using the additivity of $\Psi$ and $\Id$, we obtain the desired extension. 
\end{proof}

Theorem \ref{functordecomp} and the subsequent corollaries do not generalize to the case $c>1$. For example, if $X$ is a K3-surface, there are exact autoequivalences of $\Der(X)=\DER^b_{(2)}(X)$ without a decomposition of this form. But there is another possible generalization.\\
Note that two irreducible smooth projective curves $X$ and $Y$ are isomorphic if there are birational equivalent, i.e.\ if and only if there is an exact $k$-linear equivalence between their quotient categories $\DER^b_{(0)}(X)$ and $\DER^b_{(0)}(Y)$. 
%If we set $\DER^b_{(-1)}(X)=0$ for any variety $X$, a similar statement holds also for irreducible smooth projective varieties of dimension zero.  
In the case of surfaces Corollary \ref{isomorphic} is the dimension two generalization of this statement. This leads directly to the following questions generalizing these cases.\\

\textbf{Question.} \it Are two irreducible smooth projective varieties $X$ and $Y$ of dimension $d$ isomorphic if and only if their quotient categories $\DER^b_{(d-1)}(X)$ and $\DER^b_{(d-1)}(Y)$ are equivalent as $k$-linear triangulated categories? More general, are $X$ and $Y$ isomorphic in codimension $c<d$ if and only if $\Cohc(X)$ and $\Cohc(Y)$ are derived equivalent? \rm\\

We close this section by showing that there is no Serre functor on $\DER^b_{(1)}(X)$ for $\dim(X)\ge 2$, i.e. no $k$-linear exact autoequivalence $S$ and natural  homomorphisms
\[ \eta_{A,B}: \Hom_{(1)}(A,B) \longrightarrow \Hom_{(1)}(B,S(A))^\ast \]
inducing a non-degenerate pairing 
\begin{equation} \label{pairing} \Hom_{(1)}(A,B)\times \Hom_{(1)}(B,S(A)) \longrightarrow k.\end{equation}
Note, that we should not require that the pairing is perfect as our $\Hom$-groups might be infinite-dimensional. Let us assume the contrary. By Proposition \ref{autoequiv} there is an integer $n$ such that $S(\Coh_{(1)}(X))=\Coh_{(1)}(X)[n]$. Take a torsionfree sheaf $B$ and insert $A=B$ into equation (\ref{pairing}). Using Proposition \ref{HomExt}, we conclude $n=0$. Inserting $A=(\mbox{torsion sheaf})[-1]$ leads to a contradiction.

\newpage

\begin{appendix}
 
\section{Equivalences of quotient categories}

We will complete the proof of Proposition \ref{equivcat} by showing that $\Derc(X)$ is equivalent to $\Der(\Cohc(X))$. For the notation see Subsection \arabic{claim}.1.\\
The definition of $\Derc(X)$ and $Q$ shows that $Q:\Der(X)\longrightarrow \Derc(X)$ is the localization functor with respect to the set of morphisms $f:E\longrightarrow F$ in $\Der(X)$ with the property $C(f)\in \Dr{c+1}(X)$ which is equivalent to $\ker H^i(f), \coker H^i(f)$ $\in \Ch{c+1}(X)$ for all $i\in \mathbb{Z}$. Let us denote by $S^{c+1}$ the set of complex homomorphisms $g:E\longrightarrow F$ in $\Kom(X)$ with $\ker H^i(g), \coker H^i(g) \in \Ch{c+1}(X)$ for all $i\in \mathbb{Z}$. This set contains the set of quasi-isomorphisms. If we represent $f\in \Hom_{\Der(X)}(E,F)$ by the roof
\[ \xymatrix @C=0.5cm @R=0.5cm { & E' \ar[dl]_s \ar[dr]^g & \\ E & & F \;,} \]
we see that $C(f) \in \Dr{c+1}(X)$ if and only if $g\in S^{c+1}$. Using this and the definition of $\Der(X)$ as the localization of $\Kom(X)$ with respect to the set of quasi-isomorphisms, we obtain that $\Derc(X)$ is (isomorphic to) the localization of $\Kom(X)$ by $S^{c+1}$. In particular, a morphism in $\Derc(X)$ can be represented by a roof
\[ \xymatrix @C=0.5cm @R=0.5cm { & E' \ar[dl]_t \ar[dr]^h & \\ E & & F } \]
with $t$ and $h$ complex homomorphisms and $t\in S^{c+1}$.  
We will use this equivalent description of $\Derc(X)$ to verify the following proposition.
\begin{prop}[cf.\ Proposition \ref{equivcat}]
For any $0\le c \le d$ the naturally induced exact functor $\Der(P):\Der(X)=\Der(\Coh(X)) \longrightarrow \Der(\Cohc(X))$ factorizes over $Q:\Der(X) \longrightarrow \Derc(X)$ and the resulting functor $T:\Derc(X) \longrightarrow \Der(\Cohc(X))$ is an equivalence of triangulated categories.
\end{prop}
\[ \xymatrix @C=0.2cm @R=1.3cm { & {\Der(X)=\Der(\Coh(X))} \ar[dr]^{\Der(P)} \ar[dl]_Q & \\ {\Derc(X)} \ar[rr]^\sim_T & & {\Der(\Cohc(X))} } \]
\begin{proof} We prove the proposition in several steps. The existence of $T$ was already verified in Subsection \arabic{claim}.1. We still need to show that $T$ is fully faithful and that every object of $\Der(\Cohc(X))$ is isomorphic to some object in the image of $T$. We formulate these assertions  as three lemmas following the definition.
\end{proof}
\begin{defin}
 For every coherent sheaf $F$ let $\torssh(F)$ be the biggest subsheaf of $F$ whose support has codimension greater than $c$. For every complex $E$ of coherent sheaves we define $\torssh(E)$ to be the  subcomplex of $E$ with components $\torssh(E_n)$, where the $E_n$ are the components of $E$. It is the biggest subcomplex of $E$ which is a complex in $\Ch{c+1}(X)$.
\end{defin}
\begin{lemma}
 For every bounded complex $E$ in $\Cohc(X)$ there is a bounded complex $E'$ in $\Coh(X)$ with $\torssh(E')=0$ and an isomorphism $s:T(E') \longrightarrow E$  of complexes, where $T(E')$ is the complex $E'$ regarded as a complex in $\Cohc(X)$. In particular, every object of $\Der(\Cohc(X))$ is isomorphic to some object in the image of $T$. 
\end{lemma}
\setcounter{lemma1}{\value{prop}}
\begin{proof}
 Let us write the complex $E$ as $E_1 \xrightarrow{d_1} \ldots \longrightarrow E_{n-1} \xrightarrow{d_{n-1}} E_n \xrightarrow{d_n} E_{n+1} $. We represent $d_{n}$ by a roof in $\Coh(X)$
\[ \roof{E_n}{E'_n}{E_{n+1}}{s_n}{\tilde{d}_n} \]
with $\ker(s_n),\coker(s) \in \Ch{c+1}(X)$ and obtain an isomorphism of complexes in $\Cohc(X)$
\[ \xymatrix @C=1cm @R=1cm { E^{(1)} \ar[d]^{s^{(1)}}&:  & E_1 \ar@{=}[d] \ar[r]^{d_1} & \quad\ldots\quad \ar[r]^{d_{n-2}} & E_{n-1} \ar@{=}[d] \ar[r]^{s_n^{-1}d_{n-1}} & E'_n \ar[d]^{s_n} \ar[r]^{\tilde{d}_n} & E_{n+1} \ar@{=}[d] \\ E  &:  & E_1 \ar[r]^{d_1} & \quad\ldots\quad \ar[r]^{d_{n-2}} & E_{n-1} \ar[r]^{d_{n-1}} & E_n \ar[r]^{d_n} & E_{n+1}\;. } \]
The advantage of $E^{(1)}$ is that $\tilde{d}_n$ is a morphism in $\Coh(X)$ instead of $\Cohc(X)$. We repeat this procedure with the differential $s_n^{-1}d_{n-1}$ and obtain an isomorphism $s^{(2)}:E^{(2)} \longrightarrow E^{(1)}$ and the last two differentials of $E^{(2)}$ are morphisms of sheaves. Progressing in this way, we get an isomorphism $s^{(1)}\circ \ldots \circ s^{(n)}: E^{(n)} \longrightarrow E$ of complexes in $\Cohc(X)$ and all differentials of $E^{(n)}$ are morphisms of sheaves. However, $E^{(n)}$ does not need to be a complex in $\Coh(X)$. It is a complex in $\Cohc(X)$, i.e.\  the image of the composition of two successive differentials is contained in $\Ch{c+1}(X)$. We obtain another isomorphism of complexes
\[ \xymatrix @C=0.7cm @R=1cm { E^{(n)} \ar[d]^{s'}&:  & E^{(n)}_1 \ar@{->>}[d] \ar[r]^{d^{(n)}_1} & \quad\ldots\quad \ar[r]^{d^{(n)}_{n}} & E^{(n)}_{n+1} \ar@{->>}[d]  \\ E'=E^{(n)}/\torssh(E^{(n)})  &:  & E^{(n)}_1/\torssh(E^{(n)}_1) \ar[r]^{d'_1} & \quad\ldots\quad \ar[r]^{d'_{n}} & E^{(n)}_{n+1}/\torssh(E^{(n)}_{n+1})  } \]
in $\Cohc(X)$. The composition of two successive differentials in the complex below is zero by the construction of $E'$. Thus, $E'$ can be regarded as a complex in $\Coh(X)$ with $\torssh(E')=0$  which is mapped by $T$ onto itself as an object of $\Kom(\Cohc(X))$. The requested isomorphism is 
$ s^{(1)}\circ \ldots \circ s^{(n)}\circ s'^{-1}: T(E') \longrightarrow E$.
\end{proof}
\textbf{Remark.} Note that the isomorphism in the lemma is an isomorphism of complexes in $\Cohc(X)$ and not just a quasi-isomorphism which would be enough for the second statement of the lemma.  Furthermore, for any complex $E$ in $\Coh(X)$ the complex homomorphism $E \sur E/\torssh(E)$ is contained in $S^{c+1}$ and, thus, an isomorphism in $\Derc(X)$. In other words, we can replace any object in $\Derc(X)$ by an isomorphic complex without nontrivial subcomplexes in $\Ch{c+1}(X)$.
\begin{lemma}
 The functor $T$ is full, i.e.\  for two arbitrary objects $E,F\in \Derc(X)$ the map
\[ \Hom_{\Derc(X)}(E,F) \xrightarrow{\;\;T\;\;} \Hom_{\Der(\Cohc(X))}(T(E),T(F)) \]
is surjective.
\end{lemma}
\begin{proof} Due to the upper remark, we can assume $\torssh(E)=\torssh(F)=0$. Let us represent a morphism $f:T(E) \longrightarrow T(F)$ by a roof
 \[ \roof{T(E)}{G}{T(F)}{s}{\tilde{f}} \]
with complex homomorphisms $\tilde{f}$ and $s$ in $\Cohc(X)$ and $s$ is a quasi-isomorphism. Due to the previous lemma, we have an isomorphism $s':T(G')\longrightarrow G$ of complexes and we can replace the roof by the equivalent roof
\[ \roof{T(E)}{T(G')}{T(F)}{ss'}{\tilde{f}s'}\qquad\mbox{ with } \torssh(E)=\torssh(F)=\torssh(G')=0. \]
Thus, it is enough to show that we can `lift' every morphism represented by a complex homomorphism $f:T(E) \longrightarrow T(F)$ in $\Cohc(X)$ to a morphism $\hat{f}:E\longrightarrow F$ in $\Derc(X)$ under the assumption $\torssh(E)=\torssh(F)=0$. Furthermore, $\hat{f}$ needs to be an isomorphism if $f$ is a quasi-isomorphism. \\
If we represent every component of $f$ by a roof in $\Coh(X)$, we get the following diagram
\[ \xymatrix @C=1.5cm @R=1cm { E_1 \ar[r]^{d_1} & \quad\ldots\quad \ar[r]^{d_{n-1}} & E_n \ar[r]^{d_n} & E_{n+1} \\
E'_1 \ar[u]_{s_1} \ar[d]^{f_1} & \ldots & E'_n \ar[u]_{s_n} \ar[d]^{f_n} & E'_{n+1} \ar[u]_{s_{n+1}} \ar[d]^{f_{n+1}} \\ F_1 \ar[r]^{h_1} & \quad\ldots\quad \ar[r]^{h_{n-1}} & F_n \ar[r]^{h_n} & F_{n+1} } \]
with $\ker(s_1), \ldots, \ker(s_{n+1}), \coker(s_1), \ldots, \coker(s_{n+1}) \in \Ch{c+1}(X)$. Since \\
$\torssh(E_k)=\torssh(F_k)=0$, the morphisms $s_k$ and $f_k$ factorize over the quotient map $E'_k \sur E'_k/\torssh(E'_k)$. The commutativity of the upper diagram in $\Cohc(X)$ is still valid if we replace $E'_k$ by the quotient $E'_k/\torssh(E'_k)$. Thus, we can assume $\torssh(E'_k)=0$. In this case, $s_k$ is injective and we can regard $E'_k$ as a subsheaf of  $E_k$ with the inclusions $s_k$. We consider now the subsheaves $E''_k\mathrel{\mathop{:}}=E'_k\cap d_k^{-1}(E'_{k+1})$ of $E_k$. Since $d_{k+1}\circ d_k=0$, we get $d_k(E''_k)\subseteq E''_{k+1}$ and obtain a subcomplex $E''$ of $E$ with components $E''_k$. Using $E_k/E'_k, E_{k+1}/E'_{k+1} \in \Ch{c+1}(X)$, it is an easy exercise to show that $E_k/E''_k$ is contained in $\Ch{c+1}(X)$. Thus, the inclusion $s'':E''\longrightarrow E$ is in $S^{c+1}$ and, therefore, an isomorphism in $\Derc(X)$. The restrictions $f''_k$ of the $f_k$ to $E''_k$ form a complex homomorphism in $\Cohc(X)$, i.e.\  the image of $h_kf''_k-f''_{k+1}d_k$ is contained in $\Ch{c+1}(X)$. Due to our assumption on $F$, there are no nontrivial subsheaves of $F_{k+1}$ of this kind. This shows that the $f''_k$ form a complex homomorphism $f'':E''\longrightarrow F$ in $\Coh(X)$. The composition $\hat{f}\mathrel{\mathop{:}}=f''\circ s''^{-1}:E \longrightarrow F$ which is well defined in $\Derc(X)$ is the desired lift of $f:T(E) \longrightarrow T(F)$. \\
Finally, we have to show that for a quasi-isomorphism $f$ the lift $\hat{f}$ is an isomorphism. If $f$ is a quasi-isomorphism in $\Kom(\Cohc(X))$, its cone $C(f)$ must be zero in $\Der(\Cohc(X))$. On the other hand, we have $\Der(P)(C(f''))\cong T(C(\hat{f}))\cong C(f)=0$. Since $\Der(P)$ commutes with cohomology, the cohomology sheaves of $C(f'')$ are contained in the kernel of $P$ which is $\Ch{c+1}(X)$. Therefore, $f''\in S^{c+1}$ and $\hat{f}=f''\circ s''^{-1}$ is an isomorphism in $\Derc(X)$.  
\end{proof}
\begin{lemma}
 The functor $T$ is faithful, i.e.\  for two arbitrary objects $E,F\in \Derc(X)$ the map
\[ \Hom_{\Derc(X)}(E,F) \xrightarrow{\;\;T\;\;} \Hom_{\Der(\Cohc(X))}(T(E),T(F)) \]
is injective.
\end{lemma}
\begin{proof}
Let $\tilde{f}:E\longrightarrow F$ be a morphism with $T(\tilde{f})=0$. We represent $\tilde{f}$ by a roof
\[ \roof{E}{E'}{F}{s}{f} \]
of complex homomorphisms in $\Coh(X)$ with $s\in S^{c+1}$. Replacing $E$ and $F$ if necessary, we can assume $\torssh(E)=\torssh(F)=0$. In this case $f$ and $s$ factorize over the quotient map $E'\sur E'/\torssh(E')$. Thus, we can assume $\torssh(E')=0$ as well. It is enough to find a complex homomorphism $t'':G'' \longrightarrow E'$ in $S^{c+1}$ with $f\circ t''=0$. \\
Since $T(s)$ is an isomorphism, we have $T(f)=0$ in $\Der(\Cohc(X))$. The latter is equivalent to the existence of a quasi-isomorphism $u:G \longrightarrow T(E')$ in $\Kom(\Cohc(X))$ with $f\circ u=0$ as a complex homomorphism. Lemma A.\arabic{lemma1} provides us with a complex isomorphism $u':T(G') \longrightarrow G$ with $\torssh(G')=0$. We denote the composition $u\circ u'$ by $t:T(G')\longrightarrow T(E')$. Due to the proof of the previous lemma, we can lift $t$ to an isomorphism $\hat{t}:G' \longrightarrow E'$ in $\Derc(X)$ represented by the roof
\[ \roof{G'}{G''}{E'}{s''}{t''} \]
with $s'',t''\in S^{c+1}$ and, moreover, $s''$ is an isomorphism regarded as a complex homomorphism in $\Cohc(X)$. Since $f\circ t'' \circ s''^{-1}=0$ in $\Kom(\Cohc(X))$, we get $f\circ t''=0$ in $\Kom(\Cohc(X))$. This means that the image of $f\circ t'':G''\longrightarrow F$ is a subcomplex of $F$ in $\Ch{c+1}(X)$. Due to our assumption on $F$, we get $f\circ t''=0$ in $\Kom(\Coh(X))$ and we are done.
\end{proof}
Note that the proof does not work in the case $\altM od(X)$ instead of $\Coh(X)$ since the sheaf $\torssh(E)$ does not exist in general, e.g.\  for the sheaf $E=\bigoplus_{x\in X} k(x)$.

\newpage
\section{Direct limits of spectral sequences}

Appendix B contains the proof of Lemma \ref{limitsequence} which states the following: 
\begin{lemma}[cf.\ Lemma \ref{limitsequence}] 
For every $F,G\in \QCoh(X)$ the spectral sequences 
\[\Ho^p(U,\altE xt^q(F|_U,G|_U)) \Longrightarrow \Ext^{p+q}(F|_U,G|_U)\] 
form an inductive system over the directed set of open subsets $U\subset X$ with $\codim(X\setminus U)>c$. There is a direct limit spectral sequence converging against 
\[ E^n = \varinjlim_{\codim(X\setminus U)>c} \Ext^n(F|_U,G|_U)\]
with $E_2$-term 
\[ E_2^{p,q} = \varinjlim_{\codim(X\setminus U)>c} \Ho^p(U,\altE xt^q(F|_U,G|_U)). \]
\end{lemma}
For the notation we refer to Section \arabic{claim}. Before proving this let us recall the notion of a spectral sequence (cf.\ \cite{Gelfand-Manin} or \cite{McCleary}). A spectral sequence in an abelian category $\altA$ consists of a collection of objects $(E^{p,q}_r,E^n),\: n,p,q,r\in \mathbb{Z},\: r\ge 1$, and morphisms $d^{p,q}_r:E^{p,q}_r \longrightarrow E^{p+r,q-r+1}_r$ satisfying the following conditions:
\begin{itemize}
\item[(i)] $d^{p+r,q-r+1}_r\circ d^{p,q}_r=0$ for all $p,q,r$ and we denote the quotient \\
$\ker(d^{p,q}_r)/\im(d^{p-r,q+r-1}_r)$ by $H^{p,q}(E_r)$. 
\item[(ii)] There are isomorphisms $\alpha^{p,q}_r:H^{p,q}(E_r)\stackrel{\sim}{\rightarrow} E^{p,q}_{r+1}$ which are part of the data.
\item[(iii)] For any pair $(p,q)$ there is an $r_0$ such that $d^{p,q}_r=d^{p-r,q+r-1}_r=0$ for all $r\ge r_0$. In particular, $E^{p,q}_r \cong E^{p,q}_{r_0}=\mathrel{\mathop{:}}E^{p,q}_\infty$ for all $r\ge r_0$.
\item[(iv)] For every $n$ there is a descending filtration $\ldots \subset F^{p+1}E^n\subset F^pE^n \subset \ldots \subset E^n$ such that $\bigcap F^pE^n=0$ and $\bigcup F^pE^n=E^n$, and isomorphisms $\beta^{p,q}:E^{p,q}_\infty \stackrel{\sim}{\rightarrow} F^pE^{p+q}/F^{p+1}E^{p+q}$.
\end{itemize}
A morphism between spectral sequences $E$ and $E'$ is a collection of morphisms $f^{p,q}_r:E^{p,q}_r \rightarrow E'^{p,q}_r$ and $f^n:E^n\rightarrow E'^{n}$ compatible with all morphisms $d^{p,q}_r,\alpha^{p,q}_r,\beta^{p,q}$ and all filtrations. It is easy to see that the category of spectral sequences is additive but in general not abelian. Thus, direct and inverse limits do not exist in general. Neverless we have the following result.
\begin{lemma} \label{inductivesequence}
Let $(E_{(u)})_{u\in J}$ be inductive system of spectral sequences in the category of abelian groups, i.e. $J$ is a directed preordered set and for $v\le u$ there is a morphism $\rho^u_v:E_{(u)} \rightarrow E_{(v)}$ of spectral sequences such that $\rho^u_u=id$ and $\rho^v_w\circ \rho^u_v=\rho^u_w$ if $w\le v\le u$. If $r_0$ in the definition of the spectral sequence is bounded from above for all $u\in J$, there is a direct limit $\varinjlim_{u\in J}E_{(u)}$ of this system given by a spectral sequence with 
\[ (\varinjlim_{u\in J}E_{(u)})^{p,q}_r=\varinjlim_{u\in J} E^{p,q}_{(u)r} \quad\mbox{and}\quad (\varinjlim_{u\in J}E_{(u)})^n=\varinjlim_{u\in J} E_{(u)}^n\] 
\end{lemma}
\begin{proof} As the direct limit of an inductive system of abelian groups is an exact functor, the groups $\varinjlim_{u\in J} H^{p,q}(E_{(u)r})$ are still the cohomology groups of the limit differential $\varinjlim_{u\in J} d^{p,q}_{(u)r}$. Similarly, $\varinjlim_{u\in J}F^pE^{p+q}_{(u)}/F^{p+1}E^{p+q}_{(u)}$ is (up to isomorphism) the $p$-th quotient of the limit filtration $\varinjlim_{u\in J}F^pE^n_{(u)}$.
\end{proof}
Note that the  condition on $r_0$ in the lemma is satisfied if we only deal with spectral sequences in the first quadrant, i.e.\ $E^{p,q}_r=0$ for $p<0$ or $q<0$.\\

To proof Lemma B.1 we have to construct an inductive system of spectral sequences over the directed preordered set of open subsets $U$ of $X$ with $\codim(X\setminus U)>c$ and with $E_2$-term and $E^n$-terms as in the lemma. The existence of the inductive system is a direct consequence of the construction of the spectral sequences mentioned in the lemma. Let us repeat the construction for fixed $F$ and $G$. First of all we choose a resolution 
\[ 0\longrightarrow G \longrightarrow I^0\longrightarrow I^1 \longrightarrow I^2 \longrightarrow \ldots \]
of $G$ by injective $\ho_X$-modules $I^q$. Since restriction to an open subset is an exact functor, we obtain resolutions 
\[ 0\longrightarrow G|_U \longrightarrow I^0|_U\longrightarrow I^1|_U \longrightarrow I^2|_U \longrightarrow \ldots \]
of $G|_U$ by injective $\ho_U$-modules (cf.\ \cite{Hartshorne}, III, Lemma 6.1). The application of $\altH om(F,-)$ resp.\ $\altH om(F|_U,-)$ to these sequences yields the complexes 
\[\altH om(F,I^\bullet)=R\altH om(F,G)\quad\mbox{resp.}\quad \altH om(F|_U,I^\bullet|_U)=R\altH om(F|_U,G|_U).\] 
As $\altH om(E_1,E_2)|_U=\altH om(E_1|_U, E_2|_U)$ for two $\ho_X$-modules $E_1$ and $E_2$, we get $R\altH om(F,G)|_U=R\altH om(F|_U,G|_U)$ and, thus, $\altE xt^q(F,G)|_U=\altE xt^q(F|_U,G|_U)$ for all $q\in \mathbb{N}$ using the exactness of $|_U$ again. \\
To compute the derived functor of the composition $\Hom(F,G)=\Gamma(X,\altH om(F,G))$ we take a Cartan--Eilenberg resolution of the complex $R\altH om(F,G)$, i.e.\ a double complex $(C^{p,q})_{p,q\ge 0}$ of injective $\ho_X$-modules and a morphism $\varepsilon:\altH om(F,I^\bullet) \longrightarrow C^{0,\bullet}$ of complexes satisfying the following conditions.
\begin{itemize}
 \item[(a)] The naturally induced complexes
\begin{eqnarray*} 
&& 0 \longrightarrow \altH om(F,I^q) \stackrel{\varepsilon}{\longrightarrow} C^{0,q} \longrightarrow C^{1,q} \longrightarrow \ldots \;, \\
&& 0 \longrightarrow B^q(\altH om(F,I^\bullet)) \stackrel{\varepsilon}{\longrightarrow} B^q_{II}(C^{0,\bullet}) \longrightarrow B^q_{II}(C^{1,\bullet}) \longrightarrow \ldots \;, \\
&& 0 \longrightarrow Z^q(\altH om(F,I^\bullet)) \stackrel{\varepsilon}{\longrightarrow} Z^q_{II}(C^{0,\bullet}) \longrightarrow Z^q_{II}(C^{1,\bullet}) \longrightarrow \ldots \;, \\
&&  0\longrightarrow H^q(\altH om(F,I^\bullet))=\altE xt^q(F,G) \longrightarrow H^q_{II}(C^{0,\bullet}) \longrightarrow H^q_{II}(C^{1,\bullet}) \longrightarrow \ldots \;, 
\end{eqnarray*}
are acyclic, where $B^q(\altH om(F,I^\bullet))$ resp.\  $B_{II}^q(C^{p,\bullet})$ denotes the image of the $(q-1)$-th differential of the complex $\altH om(F,I^\bullet)$ resp.\ $C^{p,\bullet}$. Similarly, $Z$ is the kernel and $H$ denotes the cohomology. 
\item[(b)] The following short exact sequences split
\begin{eqnarray*}
 & 0 \longrightarrow B^q_{II}(C^{p,\bullet}) \longrightarrow Z^q_{II}(C^{p,\bullet}) \longrightarrow H^q_{II}(C^{p,\bullet}) \longrightarrow 0,& \\
 & 0 \longrightarrow Z^q_{II}(C^{p,\bullet}) \longrightarrow C^{p,q} \longrightarrow B^{q+1}_{II}(C^{p,\bullet}) \longrightarrow 0.&
\end{eqnarray*}
\end{itemize}
Using these properties and the fact that $C^{p,q}$ is injective, the sequence 
\begin{equation} \label{resolcoh} 0\longrightarrow H^q(\altH om(F,I^\bullet))=\altE xt^q(F,G) \longrightarrow H^q_{II}(C^{0,\bullet}) \longrightarrow H^q_{II}(C^{1,\bullet}) \longrightarrow \ldots\end{equation}
turns out to be an injective resolution of $\altE xt^q(F,G)$. Moreover, if we denote by $\tot(C^{\bullet,\bullet})$ the diagonal complex of $C^{\bullet,\bullet}$, i.e. $\tot(C^{\bullet,\bullet})^n=\bigoplus_{p+q=n}C^{p,q}$ with differential $d=d_I+d_{II}$, the complex homomorphism $\varepsilon$ induces a quasi-isomorphism $\tilde{\varepsilon}: R\altH om(F,G) \longrightarrow \tot(C^{\bullet,\bullet})$ (see \cite{Gelfand-Manin}, III, Lemma 12). Hence, $\tot(C^{\bullet,\bullet})$ is an injective resolution of $R\altH om(F,G)$ and we obtain 
\begin{equation} \label{totcompl} R\Hom(F,G)=R\Gamma(X,R\altH om(F,G))=\Gamma(X,\tot(C^{\bullet,\bullet}))=\tot(\Gamma(X,C^{\bullet,\bullet})).\end{equation}
Since $|_U$ is exact and maps injective $\ho_X$-modules to injective $\ho_U$-modules, the restriction of our Cartan--Eilenberg resolution to an open subset $U$ gives a Cartan--Eilenberg resolution of the complex $R\altH om(F|_U,G|_U)=\altH om(F|_U,I^\bullet|_U)$. \\
Now, we apply the section functor $\Gamma(X,-)$ to the Cartan--Eilenberg resolution and analogue for the restricted resolution. Since there is a natural map $\Gamma(X,E)\longrightarrow \Gamma(U,E|_U)$ for every $\ho_X$-module $E$, we obtain an inductive system of double complexes
\[ E^{p,q}_{(U)}\mathrel{\mathop{:}}=\Gamma(U,C^{p,q}|_U) \]
over the directed preordered set of all open subsets $U$ of $X$ with $\codim(X\setminus U)>c$. Using the resolution (\ref{resolcoh}) we obtain in particular the inductive system 
\begin{equation} \label{E2} \Ho^p(U,\altE xt^q(F|_U,G|_U))=H^p_I H^q_{II} E^{\bullet,\bullet}_{(U)},\end{equation}
and by equation (\ref{totcompl}) the inductive system 
\begin{equation}  \label{En} \Ext^n(F|_U,G|_U)=H^n\tot(E^{\bullet,\bullet}_{(U)}).\end{equation}
There is a functor from the category of all double complexes $E^{\bullet,\bullet}$ of abelian groups into the category of spectral sequences if abelian groups such that the spectral sequence associated to $E^{\bullet,\bullet}$ has the $E_2$-term 
\[ E_2^{p,q}=H^p_I H^q_{II} E^{\bullet,\bullet} \]
and the `limits' $E^n$ are given by 
\[ E^n=H^n\tot(E^{\bullet,\bullet}).\]
For a proof of this very technical construction we refer to \cite{Gelfand-Manin} or \cite{McCleary}. If we apply this functorial construction to our inductive system $(E_{(U)}^{\bullet,\bullet})_{\codim(X\setminus U)>c}$, we get an inductive system of spectral sequences with the correct $E_2$-term and the correct limits $E^n$ by (\ref{E2}) and (\ref{En}). This together with Lemma \ref{inductivesequence} proves Lemma B.1.

\end{appendix}

\bibliographystyle{plain}
\bibliography{Literatur}

\vfill
\textsc{\small S. Meinhardt: Mathematisches Institut, Universit\"at Bonn, Beringstr. 1, 53115 Bonn, Germany}\\
\textit{\small E-mail address:} \texttt{\small sven@math.uni-bonn.de}\\
\\
\textsc{\small H. Partsch: Mathematisches Institut, Heinrich-Heine-Universit\"at D\"usseldorf, Universit\"atsstr. 1, 40225 D\"usseldorf, Germany}\\
\textit{\small E-mail address:} \texttt{\small partsch@math.uni-duesseldorf.de}\\

\end{document}